\documentclass[10pt]{article}

\usepackage{amsfonts}
\usepackage{amsmath}
\usepackage{amssymb}
\usepackage{amsthm}
\usepackage{graphicx}
\usepackage[usenames]{color}
\definecolor{red}{rgb}{1.0,0.0,0.0}

\definecolor{blu}{rgb}{0.0,0.0,1.0}

\definecolor{gre}{rgb}{0.03,0.50,0.03}

\newtheorem{lemma}{Lemma}[section]
\newtheorem{theorem}[lemma]{Theorem}
\newtheorem{proposition}[lemma]{Proposition}
\newtheorem{corollary}[lemma]{Corollary}
\newtheorem{definition}[lemma]{Definition}
\newtheorem{problem}[lemma]{Problem}
\newtheorem{remark}[lemma]{Remark}
\newtheorem{hypothesis}[lemma]{Hypothesis}
\newtheorem{notation}[lemma]{Notation}

\def\sqr#1#2{{\vcenter{\vbox{\hrule height .#2pt \hbox{\vrule
 width .#2pt height#1pt \kern#1pt \vrule
width .#2pt} \hrule height .#2pt}}}}

\def\ds{\begin{displaystyle}}
\def\eds{\end{displaystyle}}

\def\<{\left\langle }
\def\>{\right\rangle }

\def\R{\mathbb R}
\def\N{\mathbb N}

\def\E{\mathbb E}
\def\P{\mathbb P}

\def\calb{{\cal B}}

\def\cald{{\cal D}}

\def\calf{{\cal F}}

\def\calh{{\cal H}}

\def\1{\mathbf 1}
\def\to{\rightarrow}

\begin{document}

\title{\bf Optimal portfolio choice with path dependent labor income:
the infinite horizon case}
\author{Enrico Biffis\footnote{Biffis (\texttt{e.biffis@imperial.ac.uk}) is at the Department of Finance, Imperial College Business School, London SW7 2AZ, UK.}%, and Imperial Collge Business School.},
\and Fausto Gozzi$^{\dagger}$ \and
Cecilia Prosdocimi\footnote{Gozzi (\texttt{f.gozzi@luiss.it}) and Prosdocimi (\texttt{c.prosdocimi@luiss.it}) are at the
 Dipartimento di Economia e Finanza,
 Libera Universit\`{a}  Internazionale degli Studi Sociali  "Guido Carli", Rome, Italy.}
}

\date{\today}

\maketitle

\begin{abstract}
%\noindent We consider an infinite horizon portfolio choice problem with borrowing constraints,
%in which an agent receives labor income which adjusts
%to financial market shocks in a path dependent way.
%This path-dependency
%%of the wage income process
%is the novelty of the model, and leads to an infinite dimensional
%stochastic optimal control problem.
%We solve the problem completely,
%and find explicitly the optimal controls in feedback form.
%This is possible because we are able to find an explicit solution to the associated
%infinite dimensional Hamilton-Jacobi-Bellman (HJB) equation, even if state constraints are present.
%To the best of our knowledge, this is the first infinite dimensional
%generalization of Merton's optimal portfolio problem
%where explicit solutions can be found.
%{The explicit solution allows us to study the properties of optimal strategies.}
%
%LA PARTE SOTTO VA CAMBIATA A SECONDA DELLA DISCUSSIONE
%
%{In particular we show
%how wage rigidity can modulate the negative income
%hedging demand arising from the implicit holding of risky assets in human capital, leading to richer asset allocation predictions than in the case with persistent shocks.
%We conclude by showing how the solution strategy used here can be deployed to solve other problems,
%such as the finite horizon version of the model.}
\noindent We consider an infinite horizon portfolio problem with borrowing constraints, in which an agent receives labor income which adjusts to financial market shocks in a path dependent way. This path-dependency is the novelty of the model, and leads to an infinite dimensional stochastic optimal control problem. We solve the problem completely, and find explicitly the optimal controls in feedback form. This is possible because we are able to find an explicit solution to the associated infinite dimensional Hamilton-Jacobi-Bellman (HJB) equation, even if state constraints are present. To the best of our knowledge, this is the first infinite dimensional generalization of Merton's optimal portfolio problem for which explicit solutions can be found. The explicit solution allows us to study the properties of optimal strategies
and discuss their financial implications.

%In particular we show how wage rigidity and `learning your earning' can modulate the negative income hedging demand arising from the implicit holding of risky assets in human capital, leading to richer asset allocation predictions than in the case with persistent shocks.

\end{abstract}

%\renewcommand{\baselinestretch}{1.5}
%\doublespace

\bigskip

\bigskip


\bigskip

\bigskip

\noindent


\newpage

\tableofcontents

\section{Introduction}
We consider the life-cycle optimal portfolio choice problem faced by an agent receiving labor income and allocating her wealth to risky assets and a riskless bond subject to a borrowing constraint.
The main novelty of the model is that the dynamics of labor income is path dependent, in line
with the empirical literature showing that wages adjust slowly to financial market shocks, and
income shocks have modest persistency when individuals can learn about their earning potential.
The resulting optimal control problem is infinite dimensional, and can be seen as an infinite dimensional generalization of Merton's  optimal portfolio problem.

The problem entails maximization of the expected power utility from lifetime consumption and bequest, subject to a linear state equation containing delay, as well as a state constraint, which is well known to make the problem considerably harder to solve. We are nonetheless able to exploit the structure of the model to solve it completely, and obtain
the optimal controls in feedback form (Theorem~\ref{TEO_MAIN_INF_RET}), thus allowing us to fully understand the economic implications of the setting.
To the best of our knowledge, the model presented here offers the first infinite dimensional generalization of the explicit solution to Merton's optimal portfolio problem.

Solving the problem is possible because we are able to find an explicit solution
(which we call $v$) of the associated infinite dimensional HJB equation, even if state constraints are present (Proposition~\ref{PROP_COMPARISON_FINITENESS_VAL_FUN_INF_RET}). Availability of the explicit solution, however, is not the end of the story, as proving that $v$ is indeed the actual value function and finding the feedback map (Theorems \ref{th:VERIFICATION_THEOREM_INF_RET} and \ref{th:VERIFICATIONgamma>1}) require considerable technical work. %, in view of the infinite dimensionality of the problem and the presence of state constraints.
The solution strategy developed in this paper can be used to solve other
types of problems with structure similar to the one considered here.
As such structure arises naturally in finite dimensional economic and financial
models, we think that our solution method could open the way to solving
infinite dimensional generalizations of several interesting  models.

Our interest in path-dependent labor income dynamics
originates from at least three strands of literature addressing the empirical evidence on lifecycle consumption and portfolio decisions with stochastic
labor income.
First, a common approach used to model the stochastic component of the income process is to use auto-regressive moving average (ARMA) processes (e.g., \cite{MaCurdy_1982}, \cite{ABOWD_CARD_1989}, \cite{MEGHIR_PISTAFERRI_2004}),
and several authors have shown that a parsimonious AR(1) process provides a good description of wage dynamics (see, e.g., \cite{HUBBARD_SKINNER_ZELDES_1995},\cite{MOFFITT_GOTTSCHALK_2002},
\cite{STORESLETTEN_TELMER_YARON_2004},\cite{GUVENEN_AER}).
As demonstrated by \cite{REISS_2002}, \cite{LORENZ_2006}, and
\cite{DUNSMUIR_GOLDYS_TRAN_2016}, stochastic delay
differential equations (SDDEs) can be understood, in some cases, as the weak limit of discrete time
%arise as natural continuous time counterparts of
ARMA processes: it would therefore seem natural to extend the continuous time
Merton's model to include a labor income process with a delayed dynamics
which is simple enough to deliver closed form solutions.
Second, as discussed in \cite{GUVENEN_AER}-\cite{GUVENEN_RED}, shocks in labor income
have modest persistence when heterogeneity in income growth rates is taken
into account. In particular, \cite{GUVENEN_AER} shows that allowing individuals
to learn about their income growth rate in a Bayesian way
can match several features of consumption data.
As is well known, bounded rationality and rational inattention can support the use of
moving averages instead of an optimal filter (e.g., \cite{ZHU-ZHOU}), which is exactly
what our path dependent labor income dynamics can deliver, while retaining tractability
and offering explicit solutions to the  portfolio optimization problem.
Finally,  the empirical evidence on wage rigidity (e.g., \cite{KHAN_1997}, \cite{DICKENS_ET_AL_2007},  and \cite{LEBIHAN_ET_AL_2012}, among others)
suggests that delayed dynamics may represent a very tractable way of modelling wages that adjust slowly to financial market shocks (e.g., \cite{DYBVIG_LIU_JET_2010}, section~6).
Although the model solved here is infinite horizon, it is apparent that
our findings would provide important insights in settings where an agent can retire, due to the growing relative importance of the past vs. future component of human capital as the retirement date approaches. This problem will be addressed in future work.

%More specifically in Section \ref{ACTUAL_VALUE_FUTURE_LABOR_INCOME_INF_RET}
%we will give an explicit expression for
%the human capital, in terms of the actual value of the labor income and of its past, in  Section \ref{SECTION_OPTIMIZATION_PROBLEM_INFINITE_tau_R}
%we will solve the optimization problem.
%The main results of the section are collected in Theorem \ref{TEO_MAIN_INF_RET}.
%% Symmetrically to Section  \ref{SECTION_INFINITE_tau_R}, in Section
%% \ref{ACTUAL_VALUE_FUTURE_LABOR_INCOME_FIN_RET} we will express explicitely
%% the human capital,
%% solving in
%% Section \ref{SOLUTION_OPTIMIZATION_PROBLEM_FIN_RET} the optimization problem.
%% See Theorem \ref{TEO_MAIN_FIN_RET} for the main results.
%Appendices \ref{APPENDIX_SECTION_INFINITE_tau_R} and %\ref{APPENDIX_SECTION_FINITE_tau_R}
%contain some technical proofs of few results of Section \ref{SECTION_INFINITE_tau_R}
%and
%%Section \ref{SECTION_FINITE_tau_R}
%respectively.

The structure of the paper is as follows. In the next Section \ref{Problem formulation}, we
outline the model and provide the economic motivation for the setting.
In Section~3, we rewrite the state equation (Subsection~3.1) by exploiting the representation of human wealth provided in \cite{BGP}, as well as the relevant state constraints (Subsection~3.2), both in an infinite dimensional setting where the states are Markovian.
In Section~4, the core of the paper, we solve the problem explicitly after recalling the infinite dimensional formulation (Subsection~4.1):
\begin{itemize}
  \item In Subsection~4.2, we find the explicit solution of the associated HJB equation.
  \item In Subsection~4.3, we provide a lemma to understand what happens to admissible strategies when the boundary of the constraint set is reached, a key feature in dealing with state constraints problems.
  \item In Subsections~4.4-4.5, we prove the fundamental identity and the verification theorem, which allow us to find the optimal strategies in feedback form. Here, we pay special attention to the case of risk aversion coefficient $\gamma >1$, which involves some technical complications relative to the more standard case of $\gamma \in (0,1)$.
\end{itemize}
Finally, Section~5 summarizes the main results of the paper, which are collected in
Theorem~\ref{TEO_MAIN_INF_RET}, and discusses the implications for optimal portfolio choice, as well as possible extensions of the model.
% Finally, Section~6 concludes.

\section{Problem formulation}\label{Problem formulation}

\color{black}

Consider a filtered probability space $(\Omega, \mathcal F, \mathbb F, \mathbb P)$, where we define the $\mathbb F$-adapted vector valued process $(S_0,S)$ representing the price evolution of a riskless asset, $S_0$, and $n$ risky assets, $S=(S_1,\ldots,S_n)^\top$, with dynamics
\begin{eqnarray}\label{DYNAMIC_MARKET}
\left\{\begin{array}{ll}
dS_0(t)= S_0(t) r  dt\\
dS(t) =\text{diag}(S(t)) \left(\mu dt + \sigma dZ(t)\right)\\
S_0(0)=1\\
S(0)\in {\mathbb R}^n_{+},
\end{array}
\right.
\end{eqnarray}
where we assume the following.
\begin{hypothesis}\label{hp:S}
\begin{itemize}
  \item[]
  \item[(i)]
$Z$ is a $n$-dimensional Brownian motion.
{The filtration $\mathbb F=(\mathcal F_t)_{t \ge 0}$, is the one generated by $Z$, augmented with the $\P$-null sets.}

  \item[(ii)] $\mu \in \mathbb R^n$, and the matrix $\sigma \in  \mathbb R^n \times  \mathbb R^n $ is invertible.
\end{itemize}
\end{hypothesis}

%{be invertible and satisfy $\sigma\sigma^\top>0$}, with $\mathbb R$ ($\mathbb R_+$) denoting the (non-negative) real numbers.

An agent is endowed with initial wealth $w\ge 0$, and receives
labor income $y$ until the {\color{black} random} time $\tau_{\delta}>0$, which represents the agent's time of death (see the discussion in Section~\ref{SE:DISCUSSION} for possible extensions).
%\footnote{We do not introduce a retirement date to keep the model simple. See comments in section~\ref{SE:DISCUSSION} for more details.}
We assume the following.
\begin{hypothesis}\label{hp:tau}
\begin{itemize}
  \item[]
  \item[(i)]
$\tau_{\delta}$ is independent of $Z$, and it has exponential law with parameter $\delta >0$.

  \item[(ii)] The reference filtration is accordingly
given by the enlarged filtration {\color{black} $\mathbb G := \big( \mathcal G_t \big)_{t \ge 0}$,}
where each sigma-field {\color{black} $\mathcal G_t$}
is defined as
\begin{equation*}\
%label{}
	{\color{black} \mathcal G_t}:= \cap_{u>t} \left(\mathcal F_u \vee  \sigma_g\left(\tau_{\delta}\wedge u\right)\right),
\end{equation*}
%\end{eqnarray*}
 and augmented with the $\mathbb P$-null sets. Here by $\sigma_g (U)$ we denote the sigma-field generated by the random variable $U$.
%, and where  $\mathcal F_u:=\sigma_g\left(Z(u)\right)$.}
\end{itemize}
\end{hypothesis}
Note that, with the above choice, {\color{black} $\mathbb G$} is the minimal enlargement of the Brownian filtration satisfying the usual assumptions and making $\tau_\delta$ a stopping time {\color{black}
(see \cite[Section VI.3, p.370]{Protter} or \cite[Section 7.3.3, p.420]{JYC}).
Moreover, see \cite[Proposition 2.11-(b)]{AKSAMITJEANBLANC17}, we have the following result.
If a process $A$ is ${\mathbb G}$-predictable then there exists a process $a$
which is ${\mathbb F}$-predictable and such that
\begin{equation}\label{eq:GFpred}
A(s,\omega) =  a(s,\omega),\qquad
\forall \omega\in \Omega,\; \forall s \in [0,\tau_\delta(\omega)]
\end{equation}
%(s,\omega)
%+ {\bf 1}_{[\tau(\omega),+\infty ]}(s,\omega)\hat a(\tau)(s,\omega),
%\qquad s \ge 0, \quad \omega \in \Omega,
%where $a$ is ${\mathbb F}$-predictable and $\hat a$
%Moreover, we have that for every $\mathcal G_t$-measurable r.v.
%	$\overline U$ there exists an $\mathcal F_t$-measurable r.v. $U$ coinciding with
%	$\overline U$ on $\{\tau_{\delta}>t\}$. This follows from a monotone class argument
%    after noting that every event $\overline G\in \mathcal G_t$ satisfies
%    $\overline G \cap \{\tau_{\delta}>t\} = G\cap \{\tau_{\delta} >t\}$ for some event $G\in \mathcal F_t$
%     \cite[section~7.3.3]{JYC}.}
%	, following \cite[Section~VI.3]{Protter}, using , {\color{cyan} we note that
%\cite{JYC}
%We therefore have that every $\mathbb G$-predictable process $\overline Y$ coincides with an $\mathbb F$-predictable process $Y$ on $\{\tau_{\delta} > t \}$, and by  Hypothesis \ref{hp:tau}(i) this process is unique \cite[Prop.~5.9.4.1]{JYC}.
%moreover, every $\mathbb G$-stopping time $\overline \tau$ coincides with an $\mathbb F$-stopping time $\tau$ on $\{\tau_{\delta} > t \}$.
We will therefore introduce the problem relative to the larger filtration $\mathbb G$, and then solve it by first working with pre-death processes (i.e. $\mathbb F$-predictable processes associated with $\mathbb G$-predictable processes as in \eqref{eq:GFpred}) and then finally expressing our results in terms of the original filtration $\mathbb G$ in Section~6.

\color{black}
%{\color{cyan}
{During her lifetime,} the agent can invest her resources in the riskless and risky assets, and can consume her wealth $W(t)$ at rate $c(t)\geq 0$. We denote by $\theta(t)\in \mathbb R^n$ the amounts allocated to the risky assets at each time $t\geq 0$. The agent can also purchase life insurance to reach a bequest target $B(\tau_\delta)$ at death, where $B(\cdot)\geq 0$ is also chosen by the agent. We let the agent pay an insurance premium of amount $\delta(B(t)-W(t))$ to purchase coverage of face value
$B(t)-W(t)$ for $t<\tau_\delta$. As in \cite{DYBVIG_LIU_JET_2010}, we interpret a negative face value $B(t)-W(t)<0$ as a  life annuity trading wealth at death for a positive income flow $\delta(W(t)-B(t))$ while living.
{We assume the pre-death controls $(c,B,\theta)$
to live in}
	\begin{eqnarray}\label{DEF_PI0_FIRST_DEFINITION}
	\Pi^0
	%\left(w,x_0,x_1\right)
	:= & \Bigg\{ \mathbb F-\mbox{predictable}
	%predictable
	\ c(\cdot), B(\cdot), \theta(\cdot), \ \mbox{such that:} \ c(\cdot), B(\cdot) \in L^1 (\Omega \times [0, +\infty);
	\mathbb R_{+});
	\\
	\nonumber
	&\theta(\cdot) \in L^2(\Omega \times \mathbb R; % [0, +\infty) ,
	\mathbb R^n)\Bigg\}.
	\end{eqnarray}
{Again using \cite[Proposition 2.11-(b)]{AKSAMITJEANBLANC17}, we see that the processes
\begin{equation}\label{eq:barproc}
\overline c(t)=1_{\tau_\delta \ge t} c(t),
\quad
\overline \theta(t)=1_{\tau_\delta \ge t} \theta(t),
\quad
\overline B(t)=1_{\tau_\delta \ge t} B(t),
\quad
\overline W(t)=1_{\tau_\delta \ge t} W(t)
\end{equation}
are all $\mathbb G$-predictable processes (we say that $c(\cdot),B(\cdot),\theta(\cdot),W(\cdot)$ are their pre-death counterparts), so that
the controls $(\overline c, \overline B,\overline \theta)$ live in
\begin{eqnarray}\label{DEF_PI0_FIRST_DEFINITION-G}
\overline\Pi^0
%\left(w,x_0,x_1\right)
:= & \Bigg\{\mathbb G-\mbox{predictable}
%predictable
\ \overline c(\cdot), \overline B(\cdot), \overline \theta(\cdot), \ \mbox{such that:} \ \overline c(\cdot), \overline B(\cdot) \in L^1 (\Omega \times [0, +\infty);
\mathbb R_{+});
\\
\nonumber
&\overline \theta(\cdot) \in L^2(\Omega \times \mathbb R; % [0, +\infty) ,
\mathbb R^n)\Bigg\}.
\end{eqnarray}
}
{The agent's pre-death wealth $W$ is assumed to obey the following dynamics,
\begin{align}\label{DYNAMICS_WEALTH_LABOR_INCOME-G}
\begin{split}
\left\{\begin{array}{ll}
dW(t) = & \left[W(t) r + \theta(t)^\top (\mu-r\mathbf{1})
+ y(t)-c(t)-\delta\left(B(t)-W(t)\right)\right] dt
 +\theta(t)^\top \sigma dZ(t),
%+ \left(\overline B(t)-\overline W(t)\right) dN(t),
\qquad t\ge 0.\\ [2mm]
\overline W(0) = & w,\\
\end{array}\right. \end{split}
\end{align}
%\begin{align}\label{DYNAMICS_WEALTH_LABOR_INCOME-G}
%\begin{split}
%\left\{\begin{array}{ll}
%d\overline W(t) = & \left[\overline W(t) r + \overline \theta(t)^\top (\mu-r\mathbf{1})  + \overline y(t) - \overline c(t)
%-\delta\left(\overline B(t)-\overline W(t)\right)\right] dt  \\ [2mm]
%& + \overline \theta(t)^\top \sigma dZ(t) + \left(\overline B(t)-\overline W(t)\right) dN(t), \qquad t\ge 0.\\ [2mm]
%\overline W(0) = & w,\\
%\end{array}\right. \end{split}
%\end{align}
%where $N(t):=1_{\tau_{\delta} \leq t}$ denotes the death indicator process and
where the $y(\cdot)$ is the pre-death labour income process (similarly to what we did above in \eqref{eq:barproc}, we set
$\overline y(t):=1_{\tau_{\delta}\ge t} y(t)$)
whose dynamics
is described by the following SDDE:
\begin{align}\label{DYNAMICS_LABOR_INCOME}
\begin{split}
\left\{\begin{array}{ll}
d  y(t) = & \left[ y(t) \mu_y+\int_{-d}^0 \phi(s) y(t+s) ds  \right] dt + y(t)\sigma_y^\top    dZ(t),\\[2mm]
y(0)= & x_0, \quad y(s) = x_1(s) \mbox{ for $s \in  [-d,0)$},
\end{array}\right. \end{split}
\end{align}
where $\mu_y \in  \mathbb R$, $\sigma_y \in  \mathbb R^n$,
$\mathbf 1 = (1,\dots, 1)^\top$ is the unitary vector in $\mathbb R^n$,
and the functions
$\phi(\cdot), x_1(\cdot)$ live in $L^2\left(-d,0; \mathbb R\right)$.
Existence and uniqueness of a strong solution (with $\P$-a.s. continuous paths) to the SDDE for $y$ are ensured by \cite[Theorem I.1 and Remark I.3(iv)]{MOHAMMED_BOOK_96} (see also, for a more general result, \cite[Section 3]{Rosestolato17}).
Existence and uniqueness of a strong solution to the SDE for $W$ are ensured, e.g., by the results of \cite[Chapter 5.6]{KARATZAS_SHREVE_91}.
}

%Let $N_t:=1_{\tau_\delta\leq t}$ denote the death indicator process.
%After allowing for wage income, consumption, and investment we assume that

{\begin{remark}
{\color{black} We note that the function $\phi$ allows one to modulate the contribution of different subsets of the labor income path in shaping its dynamics going forward. For example, $\phi$ could give more weight to the most recent labor income realizations relative to wage level in the more distant past. One could similarly introduce an additional delay term in the volatility  component of $y$, writing for example}
 \begin{equation*}%\label{DYN_LABOR_INCOME_DELAY_I}
 \left\{\begin{array}{ll}
\text{d}y(t) =& \left[ y(t) \mu_y+\int_{-d}^0 y(t+s)  \phi(s) \text{d}s  \right] \text{d}t \\
&\\
 &+\left[ y(t)\sigma_y^\top  + \begin{pmatrix}
 \int_{-d}^0 y(t+s)  \varphi_1(s) \text{d}s \\
 \vdots
 \\
 \int_{-d}^0y(t+s) \varphi_n(s) \text{d}s
\end{pmatrix}^\top    \right] \text{d}Z(t),\\
&\\
y(0)= & x_0, \quad y(s) = x_1(s) \mbox{ for $s \in  [-d,0)$},
\end{array}\right.
 \end{equation*}
 {\color{black}with $\varphi_1, \ldots, \varphi_n$ belonging to $L^2\left(-d,0; \mathbb R\right)$. This is the setup considered, for example, in \cite{BGP} where no control problem is considered.
In this paper, we consider path dependency in the drift of $y$ only: the extension of our results to the above general case seems possible, although it would entail an increase in complexity of
notation and technicalities. We leave it for future work.}
%\begin{align}
%\begin{split}
%\left\{\begin{array}{ll}
%dy(t) = & \left[ y(t) \mu_y+\int_{-d}^0 \phi(s) y(t+s) ds  \right] dt
%+\left[ y(t)\sigma_y^\top  + \begin{pmatrix}
% \int_{-d}^0 y(t+s)  \varphi_1(\text{d}s) \\
% \vdots
% \\
% \int_{-d}^0y(t+s) \varphi_n(\text{d}s)
%\end{pmatrix}^\top    \right] \text{d}Z(t)
%% + y(t)\sigma_y^\top +\int_{-d}^0 \varphi(s) y(t+s)    dZ(t),\\[2mm]
%y(0)= & x_0, \quad y(s) = x_1(s) \mbox{ for $s \in  [-d,0)$},
%\end{array}\right. \end{split}
%\end{align}
\end{remark}
}

{\color{black} We study the problem of
maximizing the expected utility from lifetime consumption and bequest,
%\begin{eqnarray}\label{DEF_OBJECTIVE FUNCTION_DEATH TIME}
%\mathbb E \left(\int_{0}^{+\infty} e^{-\rho t }
%\frac{c(t)^{1-\gamma}}{1-\gamma} dt
%+ e^{-\rho \tau_\delta } \frac{\big(k B(\tau_\delta)\big)^{1-\gamma}}{1-\gamma}
%\right),
%\end{eqnarray}
%over all triplets $\left(c,\theta, B\right)\in \Pi^0$
%satisfying a suitable state constraint introduced further below in
%\eqref{NO_BORROWING_WITHOUT_REPAYMENT_CONDITIONLA_MEAN}.
%{This can be seen, equivalently, as the problem of maximizing over all triplets $\left(\overline c,\overline\theta,\overline B\right)\in \overline\Pi^0$
%satisfying the above mentioned state constraint,
%the functional
\begin{eqnarray}\label{DEF_OBJECTIVE FUNCTION_DEATH TIMEbar}
\mathbb E \left(\int_{0}^{+\infty} e^{-\rho t }
\frac{\overline c(t)^{1-\gamma}}{1-\gamma} dt
+ e^{-\rho t } \frac{\big(k \overline B(t)\big)^{1-\gamma}}{1-\gamma}dN(t)
\right),
\end{eqnarray}
over all triplets $\left(\overline c,\overline\theta,\overline B\right)\in \overline\Pi^0$
satisfying a suitable state constraint introduced further below in
\eqref{NO_BORROWING_WITHOUT_REPAYMENT_CONDITIONLA_MEAN},
where we denote by $N_t:=1_{\tau_\delta\leq t}$ the death indicator process
and let parameteres $k,\gamma,\rho$ satisfy
\begin{equation}\label{eq:HPkgammarho}
k>0, \qquad \gamma \in (0,1) \cup (1, +\infty), \qquad \rho >0,
\end{equation}
an assumption that will stand throughout the paper.}

%From now on we will look at the pre-death problem (i.e. with the functional
%\eqref{DEF_OBJECTIVE FUNCTION_DEATH TIME}).}

%\begin{eqnarray}\label{DEF_OBJECTIVE FUNCTION_DEATH TIME}
%\mathbb E \left(\int_{0}^{\tau_{\delta}} e^{-\rho t }
%\frac{c(t)^{1-\gamma}}{1-\gamma} dt
%+ e^{-\rho \tau_{\delta} } \frac{\big(k B(\tau_\delta)\big)^{1-\gamma}}{1-\gamma}
%\right),
%\end{eqnarray}
%over all triplets $\left(c,\theta,B\right)\in \Pi^0$
%satisfying a suitable state constraint introduced further below.
%in
%\eqref{NO_BORROWING_WITHOUT_REPAYMENT_CONDITIONLA_MEAN-G} and
%\eqref{NO_BORROWING_WITHOUT_REPAYMENT_CONDITIONLA_MEAN}.
%{In the above, we assume
%\begin{equation}\label{eq:HPkgammarho}
%k>0, \qquad \gamma \in (0,1) \cup (1, +\infty), \qquad \rho >0.
%\end{equation}
%This will be kept throughout the whole paper.}

As the death time is independent of $Z$ and exponentially distributed, we can rewrite the objective functional in
\eqref{DEF_OBJECTIVE FUNCTION_DEATH TIMEbar}
%\eqref{DEF_OBJECTIVE FUNCTION_DEATH TIME}
as follows  (e.g., \cite[Section 3.6.2]{PHAM_BOOK_2009})
\begin{eqnarray}\label{OBJECTIVE_FUNCTION}
%\label{DEF_FUNCTIONAL_INF_RET}
\mathbb E \left(\int_{0}^{+\infty} e^{-(\rho+ \delta) t }
\left( \frac{c(t)^{1-\gamma}}{1-\gamma}
+ \delta \frac{\big(k B(t)\big)^{1-\gamma}}{1-\gamma}\right) dt
\right).
\end{eqnarray}
{\color{black} Here we work with the pre-death controls $(c,B,\theta)\in \Pi^0$
and with the pre-death state variables $(W,y)$ whose dynamics is given by the state equation:
\begin{align}\label{DYNAMICS_WEALTH_LABOR_INCOME}
\begin{split}
\left\{\begin{array}{ll}
dW(t) = & \left[W(t) r + \theta(t)^\top (\mu-r\mathbf{1})  + y(t) - c(t)
-\delta\left(B(t)-W(t)\right)\right] dt + \theta(t)^\top \sigma dZ(t) \\[2mm]
dy(t) = & \left[ y(t) \mu_y+\int_{-d}^0 \phi(s) y(t+s) ds  \right] dt + y(t)\sigma_y^\top    dZ(t),\\[2mm]
W(0) = & w,\\
y(0)= & x_0, \quad y(s) = x_1(s) \mbox{ for $s \in  [-d,0)$}.
\end{array}\right. \end{split}
\end{align}
}
%where $(c,B,\theta)\in \Pi^0
%\left(w,x_0,x_1\right)
%$.
%We note that
%%a.s. continuous trajectories for $y$ see page 40 in \cite{MOHAMMED_BOOK_96} Theorem III.4
%%We postpone to Section \ref{Section_Infinite dimensional state} the discussion
%%on the existence and uniqueness of a solution of the equation for the labor income $y$.
%existence and uniqueness of a strong solution to the SDE for $W$ are ensured, e.g.,
%by the results of \cite[Chapter 5.6]{KARATZAS_SHREVE_91}.
%}

Let us now introduce a state constraint which is natural in our context.
We first observe that, given the financial market described by \eqref{DYNAMIC_MARKET}, the pre-death state-price density of the agent obeys the stochastic differential equation
\begin{equation}\label{DYN_STATE_PRICE_DENSITY}
\left\{\begin{array}{ll}
d \xi (t)& = - \xi(t)(r +\delta) dt  -\xi(t) \kappa^\top dZ(t),\\
\xi(0)&=1.
\end{array}\right.
\end{equation}
where $\kappa$ is the market price of risk and is defined as follows (e.g., \cite{KARATZAS_SHREVE}):
\begin{equation}\label{DEF_KAPPA}
\kappa:= (\sigma)^{-1} (\mu- r \mathbf 1).
\end{equation}
We will then require the agent to satisfy the following constraint
\begin{equation}\label{NO_BORROWING_WITHOUT_REPAYMENT_CONDITIONLA_MEAN}
W(t) +   \xi^{-1}(t)\mathbb E\left( \int_t^{+\infty} \xi(u) y(u) du \Bigg\vert \mathcal F_t\right)  \geq 0,
\end{equation}
which is a no-borrowing-without-repayment constraint,
{as the second term in
\eqref{NO_BORROWING_WITHOUT_REPAYMENT_CONDITIONLA_MEAN}
represents the agent's market value of
human capital at time $t$.
In other words, human capital can be pledged as collateral,
and represents the agent's maximum borrowing capacity. The agent cannot default on his/her debt upon death, as the bequest target, $B$, is nonnegative.
We note that by ignoring the delay term (i.e., setting $\phi=0$ a.e.), the constraint reduces to $W(t)\geq - \beta^{-1}y(t)$,
with
\begin{equation}\label{DEF_BETA}
\beta:= r+\delta - \mu_y+  \sigma_y^\top \kappa,
\end{equation}
a parameter expressing the effective discount rate for labor income.   We
thus recover the borrowing constraints considered in the benchmark model of \cite{DYBVIG_LIU_JET_2010}, for example.}

Let us denote by $W^{w,x_0,x_1}\left(t; c,B,\theta\right)$ and $y^{x_0,x_1}(t)$
the solutions at time $t$ of system \eqref{DYNAMICS_WEALTH_LABOR_INCOME},
where we emphasize the dependence of the solutions on the initial conditions $(w,x_0,x_1)$
and strategies $(c,B,\theta)$. %$y\left(x_0)$
We can then define the set of admissible controls as follows:
\begin{equation}\label{DEF_PI_FIRST_DEFINITION}
\begin{split}
\Pi\left(w,x_0,x_1\right):= \Bigg\{ & c(\cdot), B(\cdot), \theta(\cdot)
\in \Pi^0,
%\left(w,x_0,x_1\right)
 \ \mbox{such that:}\\
 & W^{w,x_0,x_1}\left(t; c,B,\theta\right) +  \xi^{-1}(t)  \mathbb E\left( \int_t^{+\infty} \xi(u) y^{x_0,x_1}(u) du \Big\vert  \mathcal F_t\right)  \geq 0\,\quad \forall t \geq 0\Bigg\}.
\end{split}
\end{equation}
Our problem is then to maximize the functional given in \eqref{OBJECTIVE_FUNCTION}.
%vAs said above, this is equivalent to maximize the functional
%(which we call $J(w,x_0,x_1;c,B,\theta)$ to underline its dependence on the initial data $(w,x_0,x_1)$ and on the control variables $(c,B,\theta)$),
over all controls in
$\Pi\left(w,x_0,x_1\right)$.
%Define the value function $V(w,x_0,x_1)$ as
%\begin{align}\label{DEF_VALUE_FUNCTION_INF_RET}
%V\left(w,x_0,x_1\right):&= \sup_{\big(c,B, \theta\big) \in   \Pi\left(w,x_0,x_1 \right)} J \big( w,x_0,x_1; c,B, \theta\big) .
%\end{align}

\color{black}

We now introduce some standing  assumptions. {\color{black} Let us first define the following quantities:}
\begin{equation}\label{DEF_BETA_INFTY}
\beta_{\infty}:= \int_{-d}^0   e^{(r+\delta)s}  \phi (s)   ds,
\end{equation}
\begin{equation}\label{DEF_BETA_INFTYBAR}
\overline\beta_{\infty}:= \int_{-d}^0   e^{(r+\delta)s}  |\phi (s)|   ds.
\end{equation}
We have that $\overline\beta_{\infty}\ge \beta_{\infty}$, with the equality holding if and only if
$\phi\ge 0$ a.e.. We then introduce the following standing assumptions:
% conditions to hold throughout the rest of the paper.
%

\begin{hypothesis}\label{HYP_BETA-BETA_INFTY}
%The functions $\phi,x_1\in L^2(-d,0;\R)$ satisfy the following conditions:
%\begin{equation}\label{EQ_HYP_PHI_POSITIVE}
%\phi(s) \ge 0 \hbox{ for a.e. $s\in [-d,0]$.}
%\end{equation}
\begin{itemize}
  \item[]
  \item[(i)]
  \begin{equation}\label{EQ_HYP_BETA-BETA_INFTY}
\beta-\overline\beta_{\infty} >0.
\end{equation}

  \item[(ii)]
\begin{equation}\label{HYP_POSITIVITY_DEN_NU}
\rho + \delta -(1-\gamma) (r + \delta +\frac{\kappa^\top \kappa}{2\gamma }) >0{.}
\end{equation}
\end{itemize}
\end{hypothesis}

%\begin{hypothesis}\label{HYP_POSITIVITY_NU_INF_RET}
%The following condition is satisfied:
%\begin{equation}\label{HYP_POSITIVITY_DEN_NU}
%\rho + \delta -(1-\gamma) (r + \delta +\frac{\kappa^\top \kappa}{2\gamma }) >0,
%\end{equation}
%%The following ratio, $\nu$, is strictly positive:
%%\begin{equation}\label{DEF_nu}
%%\nu  := \frac{\gamma}{\rho + \delta -(1-\gamma) (r + \delta +\frac{\kappa^\top \kappa}{2\gamma })}>0.
%%\end{equation}
%\end{hypothesis}
\noindent
%{ == to be changed subject to relaxing
%positivity of $\phi$ ==
% \eqref{EQ_HYP_PHI_POSITIVE} ==
\begin{remark}\label{rm:hp1}
Hypothesis~\ref{HYP_BETA-BETA_INFTY} is needed
to rewrite in a convenient way constraint \eqref{NO_BORROWING_WITHOUT_REPAYMENT_CONDITIONLA_MEAN}. In particular, it implies that $\beta >0$, and hence that the effective discount rate for labor income is positive (e.g., \cite{DYBVIG_LIU_JET_2010}). This allows us to apply Theorem 2.1 of \cite{BGP}, which is recalled, in the form we need here, in Proposition \ref{PROP_CONSTRAINT_INTERPRETATION_INF_RET}.
When $\phi \ge 0$ a.e. (which also implies $\beta_{\infty}>0$),
%condition \eqref{EQ_HYP_PHI_POSITIVE} easily ensures
strict positivity of the labor income process is ensured, as shown in Proposition~\ref{pr:ypositive} below, when the initial data are positive.
\end{remark}

\begin{remark}\label{rm:hp2}
Hypothesis~\ref{HYP_BETA-BETA_INFTY}-(ii) is required to ensure that the value function is finite, as proved in Proposition~\ref{PROP_COMPARISON_FINITENESS_VAL_FUN_INF_RET}.
When $\rho+\delta>0$, such hypothesis is always satisfied if the relative risk aversion $\gamma>1$, but not in the case when $\gamma \in (0,1)$. It can actually be proved that,
when $\gamma\in (0,1)$ and
$$
\rho + \delta -(1-\gamma) (r + \delta +\frac{\kappa^\top \kappa}{2\gamma }) <0,
$$
the value function is infinite; for example, see \cite{FreniGozziSalvadori06} for the deterministic case.
%the value function is finite and then that
%Lemma \ref{LEMMA_LIMIT_AT_INFTY_GUESS_VALUE_FUNCTION_INFINITE_RETIREMENT}, which is
%fundamental in proving the Verification Theorem \ref{THM_VERIFICATION_THEOREM_INF_RET}.
%which will ensure finiteness of the value function
%(see Remark \ref{REM_FINITENESS_VALUE_FUNCTION}).
\end{remark}

%The following proposition will be useful in the sequel.

Even if it is not necessary to solve the problem, it is useful to provide conditions guaranteeing the positivity of the labor income process.
One is given in the following proposition.

\begin{proposition}\label{pr:ypositive}
Let $y(t)=y^{x_0,x_1}(t)$ be the solution at time $t$ of the second equation of system \eqref{DYNAMICS_WEALTH_LABOR_INCOME}
with initial data $x_0\in \R$, $x_1\in L^2(-d,0;\R)$.
Defining
\begin{eqnarray}
	E(t)&:=& e^{(\mu_y - \frac{1}{2} \sigma_y^\top \sigma_y) t + \sigma_y^\top Z(t)}\\
	I(t)&:= &\int_0^t     E^{-1}(u) \left(\int_{-d}^0 \phi(s)y(u+s) ds\right) du,
\end{eqnarray}
we have
\begin{eqnarray}\label{VARIATION_CONSTANTS_FORMULA}
y(t)= E(t)\big(x_0 + I(t)\big).
\end{eqnarray}
Moreover, if $x_0>0$, $x_1\ge 0$ a.e. and $\phi\ge 0$ a.e., then $y(t)>0$ must hold  $\P$-a.s..
\end{proposition}
\begin{proof}
Expression \eqref{VARIATION_CONSTANTS_FORMULA},
immediately follows by the stochastic variation of constants formula
(see, e.g., \cite{Bonaccorsi99}, Theorem 1.1). Concerning the next part of
the proposition, let $x_0>0$, $x_1\ge 0$ a.e., and $\phi\ge 0$ on $[-d,0]$ a.e..
%which is ensured by Hypothesis~\ref{HYP_POSITIVITY_NU_INF_RET}.
Let $\tau:=\inf\{t\ge 0: \; y(t)=0\}$. Since $x_0>0$ and since $y$
is continuous $\P$-a.s., then it must be $\tau>0$, $\P$-a.s..
Moreover, assume that $\tau<+\infty$ in a set $\Omega_0\subseteq \Omega$ of positive probability.
Then, from \eqref{VARIATION_CONSTANTS_FORMULA} and the fact that $E(t)>0$ for every $t\ge 0$, we immediately get $I(\tau)=-x_0$ in $\Omega_0$.
Since $x_1\ge 0$ a.e., $\phi\ge 0$ a.e., and $y(t)\ge 0$ $\P$-a.s.
where $t \le \tau$, a contradiction follows.
\end{proof}

%Due to the above result, we will find it convenient to use the following assumption only in some cases.
%
%\begin{hypothesis}\label{HYP_x_POSITIVE}
%We have $x_0>0$ and
%\begin{equation}\label{EQ_HYP_x_POSITIVE}
%x_1(s) \ge 0, \phi(s)\ge 0 \hbox{ for a.e. $s \in[-d,0]$}.
%\end{equation}
%\end{hypothesis}
%%
%%While Hypotheses \ref{HYP_BETA-BETA_INFTY} and \ref{HYP_POSITIVITY_NU_INF_RET}.
%%are always assumed and hence not repeated each time, Hypothesis \ref{HYP_x_POSITIVE}
%%will not be used unless explicitly declared.
%
%
%
%{ == TO EB: CHECK THE REMARK BELOW
%\begin{remark}\label{rm:hp3}
%Hypothesis~\ref{HYP_x_POSITIVE} is consistent, for example, with applied papers estimating the individual specific component of life-cycle earnings being positive linear in experience \cite{GUVENEN_AER}. %This is clearly not always the case.
%%Relaxing this assumption (hence allowing for the possibility of negative valued contributions to the labor income drift), while preserving strict positivity of $y$, is left for future research.
%\end{remark}
%}
%
%
%From  now on we will assume that Hypotheses \ref{hp:S}, \ref{hp:tau} and \ref{HYP_BETA-BETA_INFTY} are satisfied without repeating it.
%We will mention when Hypothesis \ref{HYP_x_POSITIVE} or others will be needed.
%

\color{black}

\section{Reformulation of the problem}
\label{Section_Mathematical_tools}
In this section we rewrite the problem in a form which is then solved in the subsequent Section \ref{SECTION_INFINITE_tau_R}.
In Subsection \ref{SSE:MARKOV} we show how to rewrite the stochastic delay equation for the labor income as a Markov SDE in a Hilbert space. In Subsection
\ref{Rephrasing the no-borrowing constraint} we show how to rewrite the no borrowing constraint \eqref{NO_BORROWING_WITHOUT_REPAYMENT_CONDITIONLA_MEAN}
in a way which is suitable for our needs.
%In Subsection \ref{SE:PROBLEMREWRITTEN} we rewrite the problem for the readers

\subsection{Reformulating the SDE for the labor income}
\label{SSE:MARKOV}

We aim to solve the stochastic optimal control problem introduced in the previous section by using the dynamic programming method.
The state equation for the labor income $y$ is a stochastic delay differential equation, and hence $y$ is not Markovian, and the same applies to the state $(W,y)$ of the control problem.
Thus, the dynamic programming principle
in its standard formulation does not apply.
As usual, (see on this e.g. \cite{VINTER75},
\cite{CHOJNOWSKA-MICHALIK_1978} or the books \cite[Section 0.2]{DAPRATO_ZABCZYK_RED_BOOK}
\cite[Section 2.6.8]{FABBRI_GOZZI_SWIECH_BOOK}),
it is convenient to reformulate the problem in an
infinite dimensional Hilbert space,
where the Markov property of the state holds,
and which takes into account both the present and the past values of the states.
%We thus need to consider an equivalent control problem,
%with state $(W,X)$ taking values in $\mathbb R \times \mathcal H$,
%where $\mathcal H$ is the Hilbert space
To be precise, let us introduce the Delfour-Mitter Hilbert space $M_2$
(see e.g. \cite[Part II - Chapter 4]{BENSOUSSAN_DAPRATO_DELFOUR_MITTER}):
\begin{equation*}
M_2 := \mathbb R \times L^2\big( -d, 0; \mathbb R \big),
\end{equation*}
with inner product, for $x=(x_0, x_1), y=(y_0, y_1) \in M_2$, defined as
$$\langle x,y \rangle_{M_2}:= x_0 y_0 + \langle x_1, y_1 \rangle_{L^2},$$
where
$$\langle x_1, y_1 \rangle_{L^2}:=\int_{-d}^0 x_1(s) y_1(sd) ds.$$
(For ease of notation, we will drop below the subscript $L^2$ from the inner product of such space, writing simply $\langle x_1, y_1\rangle$).
To embed the state $y$ of the original problem
in the space $M_2$ we now introduce the linear operators $A$ (unbounded)
and $C$ (bounded). Define the domain $\mathcal D(A)$
as follows
%\footnote{The Sobolev space  $ W^{1,2}\big( -d, 0; \mathbb R \big)$ is defined as
%	\begin{eqnarray*}
%		W^{1,2}\big( -d, 0; \mathbb R \big)
%		:=   \big\{  u \in L^2(-d, 0; \mathbb R ): \quad \exists \, g \in L^2(-d, 0; \mathbb R ) \mbox{ such that } \\
%		\int_{ [-d, 0]} u \psi^{\prime} = -  \int_{ [-d, 0]} g \psi \quad \mbox{ for all $\psi \in \mathcal C_c(-d,0; \mathbb R )$}
%		\big\},
%	\end{eqnarray*}
%	where
%	$\mathcal C_c(-d,0; \mathbb R )$ denotes the space of continuous functions with compact support in $(-d,0)$;
%	see \cite[Chapter~8]{BREZIS_BOOK}.}
\begin{equation*}
\mathcal D(A) :=\left\{(x_0,x_1) \in M_2: x_1(\cdot) \in
W^{1,2}\left( [-d, 0]; \mathbb R \right), x_0 = x_1(0)\right\}.
\end{equation*}
The operator $A: \mathcal D (A) \subset M_2 \rightarrow M_2$ is then defined as
\begin{equation}
\label{DEF_A}
A(x_0,x_1) := \left(\mu_y x_0 +\langle \phi,x_1 \rangle, x_1^{\prime}  \right),
\end{equation}
with $\mu_y,\phi$ appearing in equation (\ref{DYNAMICS_WEALTH_LABOR_INCOME}).
The operator $C:M_2 \rightarrow  \mathbb R^n \times L^2(-d,0; \mathbb R)$
is bounded and defined as
\begin{equation*}
C(x_0,x_1):= \left(x_0  \sigma_y , 0\right),
\end{equation*}
where $\sigma_y $ shows up in (\ref{DYNAMICS_WEALTH_LABOR_INCOME})
and where $0$ in the expression above stays for the null function in $L^2(-d,0; \mathbb R)$.
Proposition A.27 in \cite{DAPRATO_ZABCZYK_RED_BOOK} shows that $A$ generates a strongly continuous semigroup in $M_2$.

Consider the following stochastic differential equation in $M_2$
\begin{eqnarray}\label{INFINITE_DIMENSIONAL_STATE_EQUATION}
\begin{split}
\left\{\begin{array}{ll}
d X(t)& = A X(t) dt+( C X(t))_0^\top dZ(t),\\
X_0(0) &= x_0,\\
X_1(0)(s)& = x_1(s) \mbox{ for $s \in  [-d,0)$},
 \end{array}\right. \end{split}
\end{eqnarray}
where $( C X(t))_0^\top$ denotes the transpose of the $\mathbb R^n$-component
of the operator $C $ applied to $X(t)$.
The equation above admits a (mild) solution in $M_2$,
as ensured by Theorem 7.2
in \cite{DAPRATO_ZABCZYK_RED_BOOK},
(see the same reference for the definition of mild solution).
Denote the solution of (\ref{INFINITE_DIMENSIONAL_STATE_EQUATION}) with $X(\cdot)$
(or $X^x(\cdot)$ if we want to underline its dependence on the initial condition).
Note that $X$ is Markovian.
Theorem 3.9 and Remark 3.7 in \cite{CHOJNOWSKA-MICHALIK_1978}
(see also \cite{GOZZI_MARINELLI}, \cite{FabbriFederico14})
show that we can identify the solution $X(t)= \left(X_0(t), X_1(t) \right)$ of (\ref{INFINITE_DIMENSIONAL_STATE_EQUATION})
 with the couple $\big(y(t), y (t+s)_{\mid s \in [-d,0)} \big)$,
where $y(t)$ is the solution of the second equation in (\ref{DYNAMICS_WEALTH_LABOR_INCOME}).

{For the reader's convenience we rewrite the system
\eqref{INFINITE_DIMENSIONAL_STATE_EQUATION}
decoupling the two components of $X$.
\begin{eqnarray}\label{INFINITE_DIMENSIONAL_STATE_EQUATIONdecoupled}
\begin{split}
\left\{\begin{array}{ll}
d X_0(t)& = \left[\mu_y X_0(t) + \<\phi,X_1(t)\> \right]dt+
X_0(t)\sigma_y^\top dZ(t),\\
d X_1(t)& = \frac{\partial}{\partial s}X_1(t) dt,\\
X_0(0) &= x_0,\\
X_1(0)(s)& = x_1(s) \mbox{ for $s \in  [-d,0)$},
 \end{array}\right. \end{split}
\end{eqnarray}
We also clarify, for the reader's convenience, that the derivative of $X_1$ must be intended, in general, in distributional sense.}

%\noindent Consequently (see, for a similar example, \cite{GOZZI_MARINELLI}), the control problem with state variables $(W,y)$ we aim to solve,
%is equivalent to the following control problem.
%
%The state variables are $(W, X)=(W, X_0,X_1)$,
%where $W$ solves the first of \eqref{DYNAMICS_WEALTH_LABOR_INCOME} and
%$X$ solves (\ref{INFINITE_DIMENSIONAL_STATE_EQUATION}).
%The control variables are still $(c,B,\theta)$.
%%\cite{BENSOUSSAN_DAPRATO_DELFOUR_MITTER} for the deterministic case and \cite{GOZZI_MARINELLI} for the stochastic case).
%%We will use equivalently the notation $\left(W(t),y(t), y (t+s)_{\mid s \in [-d,0)} \right)$
%%and $\left(W(t),X_0(t), X_1(t) \right)$.
%The set of admissible controls is $\Pi\left(w,x_0,x_1\right)$ as given in \eqref{DEF_PI_FIRST_DEFINITION}.
%The functional to maximize is the one of \eqref{OBJECTIVE_FUNCTION}, which we call $J(w,x;c,B,\theta)$ to underline its dependence on the initial data $(w,x)=(w,(x_0,x_1))\in \R\times M_2$ and on the control variables $(c,B,\theta)$).
%The value function $V(w,x)=V(w(x_0,x_1))$ is then defined as
%\begin{align}\label{DEF_VALUE_FUNCTION_INF_RET}
%V\left(w,x\right):&=
%\sup_{\big(c,B, \theta\big) \in   \Pi\left(w,x_0,x_1 \right)} J \big( w,x; c,B, \theta\big).
%\end{align}
%

 \medskip

%\color{black}

We will need the following standard result on the adjoint operator of $A$.
%defined in the following proposition.
\begin{proposition}\label{pr:adjoint}
The adjoint of $A$ is $A^*: \mathcal D(A^*)\subset M_2 \longrightarrow M_2$,
defined as follows:
\begin{equation}\label{EQ_DOMAINADJOINT_OPERATOR_A*}
\mathcal D(A^*):=\left\{(y_0,y_1) \in M_2: y_1 (\cdot) \in
W^{1,2}\left( [-d, 0]; \mathbb R \right),\, y_1(-d) = 0\right\}{,}
\end{equation}
\begin{equation}\label{EQ_ADJOINT_OPERATOR_A*}
A^*(y_0,y_1):= \left(\mu_y y_0 +y_1(0), -  y_1^{\prime} + y_0 \phi \right){.}
\end{equation}
%is the adjoint operator of $A$.
\end{proposition}
\begin{proof}
We provide a sketch of the proof for the reader's convenience.
We have to check that $\mathcal D (A^*)$ coincides with the set
of points $(y_0,y_1 )\in M_{{2}}$ such that
\begin{equation}\label{eq:adjointproof1}
|\langle A(x_0,x_1), (y_0,y_1)  \rangle_{M_2} |\le M |(x_0,x_1)|_{M_2}
\quad\forall (x_0,x_1 )\in \mathcal D (A){.}
\end{equation}
Moreover we have to check that
for any $(x_0,x_1 )\in \mathcal D (A)$, $(y_0,y_1 )\in \mathcal D (A^*)$,
\begin{equation}\label{eq:adjointproof2}
\langle A(x_0,x_1), (y_0,y_1)  \rangle_{M_2}  =
\langle (x_0,x_1), A^* (y_0,y_1)  \rangle_{M_2} .
\end{equation}
%We simply check the second one for brevity.
By \eqref{DEF_A} we have, for all $x\in \mathcal D (A)$,
\begin{equation}\label{DEF_A_INNERPRODUCT}
\langle A(x_0,x_1), (y_0,y_1)  \rangle_{M_2} =   \mu_y x_0 y_0+y_0 \langle \phi , x_1\rangle+ \langle  x_1^{\prime}, y_1 \rangle.
\end{equation}
It can be proved that \eqref{eq:adjointproof1}
holds only if $y_{{1}}\in W^{1,2}\left( [-d, 0]; \mathbb R \right)${.}
Then, using integration by parts we get
\begin{equation}\label{eq:A*new}
\langle  x_1^{\prime}, y_1 \rangle=
 x_1(0) y_1(0) - x_1(-d) y_1(-d)-\< x_1, y_1^{\prime} \>.
% =x_0 y_1(0) - x_1(-d) y_1(-d)-\< x_1, y_1^{\prime} \>,
\end{equation}
From this we get that \eqref{eq:adjointproof1} also implies  $y_1(-d)=0$, hence
$y\in\mathcal D (A^*)$.
When $x\in\mathcal D (A)$ and $y\in\mathcal D (A^*)$ we get
from \eqref{DEF_A_INNERPRODUCT} and \eqref{eq:A*new}
\begin{equation}\label{DEF_A_INNERPRODUCTNEWbis}
\langle A(x_0,x_1), (y_0,y_1)  \rangle_{M_2} =
\mu_y x_0 y_0+y_0 \< \phi , x_1\>
+x_1(0) y_1(0) -\< x_1, y_1^{\prime}\>
\end{equation}
and \eqref{eq:adjointproof1} is satisfied thanks to the boundary condition
$x_0=x_1(0)$.
Finally \eqref{eq:adjointproof1} follows by straightforward computations.
\end{proof}

\subsection{Rephrasing the no-borrowing constraint}
\label{Rephrasing the no-borrowing constraint}
In this section we express the no-borrowing constraint at each time $t\geq 0$ in terms of $X_0(t)$ and $X_1(t)$ given in the previous subsection. To do so,
we introduce the constant $g_{\infty}>0$ and the function $h_{\infty}: [-d,0]\longrightarrow \mathbb R$ defined as follows:
 \begin{eqnarray}\label{DEF_g_infty_h_infty}
 \left\{
\begin{array}{ll}
g_{\infty} &:= \dfrac{1}{\beta- \beta_{\infty} },\\
\\
h_{\infty} (s)&: = g_{\infty}    { \displaystyle %\frac{1}{\beta- \beta_{\infty} }
	\int_{-d}^s   e^{-(r+\delta)(s-\tau)} \phi (\tau) d \tau,}
\end{array}
\right.
\end{eqnarray}
with $\beta $ and $\beta_{\infty}$ defined in (\ref{DEF_BETA}) and (\ref{DEF_BETA_INFTY}) respectively.
\begin{lemma}\label{LEMMA_DIFF_EQ_G_INFTY_H_INFTY}
For a.e. $s \in [-d,0]$, the function $h_{\infty}$ defined in (\ref{DEF_g_infty_h_infty}) is differentiable and it satisfies
\begin{eqnarray}\label{EQ_g_h_infty}
\left\{
\begin{array}{ll}
 h_{\infty}^{\prime} (s) &=  g_{\infty}  \phi (s) -(r+\delta) h_{\infty}(s)  ,\\
  h_{\infty} (0)&  = \beta g_{\infty} -  1.
\end{array}
\right.
\end{eqnarray}
Moreover $(g_{\infty},h_{\infty}) \in \mathcal D(A^*)$.
\end{lemma}
\begin{proof}
%\textit{of Lemma \ref{LEMMA_DIFF_EQ_G_INFTY_H_INFTY}}
	By definition of $\beta_{\infty}$ we have
	$$  h_{\infty} (0)  =\beta_{\infty}g_{\infty}  ,  $$
	and therefore
	\begin{equation*}
	\beta g_{\infty}  -  h_{\infty} (0)  = \beta g_{\infty} - \beta_{\infty}g_{\infty}  =1,
	\end{equation*}
	thus $h_{\infty}$ satisfies the terminal condition in (\ref{EQ_g_h_infty}).
Differentiability a.e. of $h_{\infty}(s)$ follows by standard differentiability of
integral functions. Differentiating we then get
	\begin{eqnarray*}
		\begin{split}
			h_{\infty}^{\prime} (s)
			= -(r+\delta) h_{\infty}(s) + g_{\infty}  \phi (s) ,
		\end{split}
	\end{eqnarray*}
	and (\ref{EQ_g_h_infty}) then follows.
	Let us now check that $h_{\infty} \in L^2 (-d,0, \mathbb R)$:
	\begin{multline*}
	\int_{-d}^0 h_{\infty}^2(s) ds =
	\frac{1}{(\beta - \beta_{\infty})^2}\int_{-d}^0 \Big(   \int_{-d}^s e^{-(r+\delta)(s-\tau)}  \phi(\tau)d \tau\Big)^2 ds \\
	\leq   \frac{1}{(\beta - \beta_{\infty})^2} \int_{-d}^0   \int_{-d}^s e^{-2(r+\delta)(s-\tau)}  \phi^2(\tau)d \tau ds
	\leq \frac{1}{(\beta - \beta_{\infty})^2}  \int_{-d}^0   \int_{-d}^s  \phi^2(\tau)d \tau ds\\
	\leq   \frac{1}{(\beta - \beta_{\infty})^2}\int_{-d}^0   \int_{-d}^0  \phi^2(\tau)d \tau ds
	=  \frac{d}{(\beta - \beta_{\infty})^2} \|\phi  \|_2^2 <+\infty,
	\end{multline*}
	where the first inequality follows by Jensen's inequality.
Using (\ref{EQ_g_h_infty}) we then immediately get that
$h_{\infty}\in W^{1,2}\left( [-d, 0]; \R \right)$.
Finally, since $h_{\infty}(-d)=0$ we get
$(g_{\infty}, h_{\infty})$ is in $\mathcal D(A^*)$.
\end{proof}

The following Proposition, which is a direct consequence of Theorem~2.1 of \cite{BGP},
 provides an explicit expression for the market value of human capital.
%$$ \xi(t)^{-1} \mathbb E\left( \int_t^{\infty} \xi(\tau) %y(\tau) d\tau \Bigg\vert \mathbb F^t\right).
%$$
%By \cite{} [ADD REF], we can set set
%\begin{equation}\label{DEF_H_INFTY}
%H_{\infty}(t):=\xi(t) \big( g_{\infty}X_0(t)+ \langle %h_{\infty}, X_1(t) \rangle\big),
%\end{equation}
%and obtain the following following result.
 \begin{proposition}\label{PROP_CONSTRAINT_INTERPRETATION_INF_RET}
Let $X(t)$ solve \eqref{INFINITE_DIMENSIONAL_STATE_EQUATION} and $W$
solve the first of \eqref{DYNAMICS_WEALTH_LABOR_INCOME} with $X_0(t)$ in place of $y(t)$.
Let $\xi(t)$ solve (\ref{DYN_STATE_PRICE_DENSITY}).
Then, the market value of human capital admits the following representation
\begin{equation}\label{eq:BGP}
\xi(t)^{-1 }\mathbb E\left( \int_t^{+\infty} \xi(u) X_0(u) du \Bigg| \mathcal F_t\right) = g_{\infty}X_0(t)+ \langle h_{\infty}, X_1(t) \rangle
%H_{\infty}(t)
\qquad \mbox{  $\mathbb P $-a.s.}.
\end{equation}
\end{proposition}
\begin{proof}
See \cite[Theorem~2.1]{BGP}.
\end{proof}
%
%\noindent

%By the above proposition we can rewrite constraint \eqref{NO_BORROWING_WITHOUT_REPAYMENT_CONDITIONLA_MEAN}
%as  %Proposition~\ref{PROP_CONSTRAINT_INTERPRETATION_INF_RET},
%\eqref{NO_BORROWING_WITHOUT_REPAYMENT_EXPLICIT_INF_RET}.
Expression \eqref{eq:BGP} shows that the market value of human capital can be decomposed into two terms: one capturing the current market value of the past trajectory of labor income over $[t-d,t]$, and one capturing the current market value of the future labor income stream (see the discussion in \cite{BGP}). This distinction will be important when interpreting the solution to our optimization problem.

Proposition \ref{PROP_CONSTRAINT_INTERPRETATION_INF_RET} implies that
the original constraint \eqref{NO_BORROWING_WITHOUT_REPAYMENT_CONDITIONLA_MEAN}
can be reformulated as follows:
\begin{equation}\label{NO_BORROWING_WITHOUT_REPAYMENT_EXPLICIT_INF_RET}
W(t) +g_{\infty}X_0(t)+ \langle h_{\infty}, X_1(t) \rangle  \geq0 \mbox{   for all $t$.}
\end{equation}

\begin{notation}\label{not1}
We note that, when $t=0$, the above implies that the initial datum $(w,x)\in \calh:=\R\times M_2$ must belong to the half-space
\begin{equation}\label{eq:defH+}
    \calh_+=\left\{(w,x)\in \calh:\;
    w +g_{\infty}x_0+ \< h_{\infty}, x_1 \> \geq 0 \right\}
\end{equation}
It is also convenient to introduce the open half-space
\begin{equation}\label{eq:defH++}
    \calh_{++}=\left\{(w,x)\in \calh:\;
    w +g_{\infty}x_0+ \< h_{\infty}, x_1 \> > 0 \right\}
\end{equation}
Moreover, defining the linear map $\Gamma_\infty:\calh \to \R$ as
\begin{equation}\label{eq:defGamma}
 \Gamma_{\infty}(w,x):= w+g_{\infty} x_0+\langle h_{\infty},x_1 \rangle,
\end{equation}
we have $\calh_+=\left\{\Gamma_\infty \ge 0\right\}$,
$\calh_{++}=\left\{\Gamma_\infty > 0\right\}$.
We observe that $\calh_+$ contains the cone of positive functions in $\calh$ if and only if $\phi \ge 0$ a.e..
\end{notation}

\section{Solving the problem}
\label{SECTION_INFINITE_tau_R}

\subsection{Statement of the reformulated problem}
\label{SSE:PBREFORMOLATED}

Using the results of the previous Section \ref{Section_Mathematical_tools},
the problem exposed in Section \ref{Problem formulation}
can be reformulated as follows.

%\color{black}

\begin{problem}\label{Problem1}
The state space is $\calh:=\R\times M_2$. The control space is
$U:=\R_+\times \R_+\times \R^n$.
The state equation is
\begin{eqnarray}\label{DYN_W_X_INFINITE_RETIREMENT_II}
\begin{split}
\left\{\begin{array}{ll}
dW(t) =&  \left[ (r+\delta) W(t)+ \theta^\top(t) (\mu-r \mathbf 1)  +  X_0(t) - c(t) - \delta B(t) \right] dt  \\
  &+  \theta^\top (t) \sigma dZ(t),\\
  dX(t) =& A X(t) dt + \big(C X(t) \big)^\top dZ_t,\\
  W(0)=& w,\\
X_0(0) =& x_0,\quad \quad
 X_1(s) = x_1(s) \mbox{ for $s \in  [-d,0)$}{.}
  \end{array}\right. \end{split}
\end{eqnarray}
Denote by $W^{w,x}(s;c,B, \theta)$ the solution at time $s$ of
the first equation in (\ref{DYN_W_X_INFINITE_RETIREMENT_II})
starting at time $0$ in $\left(w,x\right)$ and following the strategy $\left(c,B,\theta \right)$, and by $X^{x} (s)$ the solution at time $s$ of the second equation in (\ref{DYN_W_X_INFINITE_RETIREMENT_II}), starting at time $0$ in $x$.
The set of admissible controls is (see (\ref{DEF_PI_FIRST_DEFINITION})
and Proposition \ref{PROP_CONSTRAINT_INTERPRETATION_INF_RET}):
\begin{equation*}\label{DEF_PI_SECOND_DEFINITION}
\begin{split}
\Pi\left(w,x_0,x_1\right)=\Pi\left(w,x\right)= \Bigg\{ &  \mathbb F-\mbox{adapted} \ c(\cdot), B(\cdot), \theta(\cdot), \ \mbox{such that:} \ c(\cdot), B(\cdot) \in L^1 (\Omega \times [0, +\infty),\mathbb R_{+}); \\ &\theta(\cdot) \in L^2(\Omega \times [0, +\infty) , \mathbb R^n);\\
 & W^{w,x}\left(t; c,B,\theta\right) +    g_{\infty}X^x_0(t)+ \langle h_{\infty}, X^x_1(t) \rangle \geq 0\,\quad \forall t \geq 0\Bigg\},
\end{split}
\end{equation*}
(the last constraint being equivalent to ask
$(W(t),X(t))\in  \calh_+$ for every $t \ge 0$),
find a strategy $\left(c,B,\theta \right)\in \Pi\left(w,x_0,x_1\right)$
maximizing the functional
\begin{equation}\label{DEF_J_INF_RET}
J \big(w,x; c,B, \theta\big) := \mathbb E \left(\int_{0}^{+\infty} e^{-(\rho+ \delta) t }
\left( \frac{c(t)^{1-\gamma}}{1-\gamma}
+ \delta \frac{\big(k B(t)\big)^{1-\gamma}}{1-\gamma}\right) dt
\right),
\end{equation}
assuming Hypotheses \ref{hp:S}, \ref{hp:tau} and \ref{HYP_BETA-BETA_INFTY}.
\end{problem}
Note that the functional $J$ may possibly take value $-\infty$ (e.g. when $\gamma>1$ and both $c(\cdot)$ and $B(\cdot)$ are both identically zero) or $+\infty$ (when $\gamma \in (0,1)$ but, thanks to Hypothesis \ref{HYP_BETA-BETA_INFTY},
this will be proved to be impossible, see Corollary \ref{cr:FINITENESS_VALUE_FUNCTION}
and Proposition \ref{pr:Vfinitegamma>1}).
%in the subsequent Proposition
%\ref{PROP_COMPARISON_FINITENESS_VAL_FUN_INF_RET}).

We solve the problem by using the dynamic programming method.
Define, for $(w,x) \in \calh_+$, the value function $V(w,x)$ as
\begin{align}\label{DEF_VALUE_FUNCTION_INF_RET}
V\left(w,x\right):&= \sup_{\big(c,B, \theta\big) \in   \Pi\left(w,x \right)} J \big( w,x; c,B, \theta\big).
\end{align}
Similarly to what we noted above for the functional $J$ we see that, up to now,
$V$ may possibly take the values $-\infty$ or $+\infty$ in $\calh_+$.

\begin{notation}\label{not:shorthand}
Sometimes, given an initial point $(w,x)\in \calh_+$ and an admissible strategy
$(c,B,\theta)\in \Pi(w,x)$, for readability, we will use the shorthand notations:
$$
\pi:=(c,B,\theta)
$$
and
\begin{align*}
 W_{\pi}(s) :=  W^{w,x}(s;c,B,\theta), \qquad
 X(s ) :=  X^{x} (s),
\end{align*}
shrinking the dependence on the controls and omitting
the dependence on the initial conditions.
\end{notation}

\bigskip

\subsection{The HJB equation and its explicit solution}
\label{SSE:HJBEXPLICIT}

\begin{notation}
Let $p= (p_1,p_2)$
be a generic vector of $\calh=\mathbb R \times M_2$,
and let $S(2)$ denote the space of real symmetric matrices of dimension $2$, and
$P$ an element of $S(2)$, with
\begin{equation*}\label{NOTATIONS_P}
P=\left(\begin{array}{cc}
P_{11}&  P_{12}  \\
 P_{2 1}  &  P_{22}
\end{array} \right){.}
\end{equation*}
For any given function $u: \calh \longrightarrow \mathbb R$,
we denote by $Du = \left(u_w,u_{x}\right)= \left(u_w,(u_{x_0},u_{x_1})\right)\in \calh$
its gradient and by
$$D^2_{wx_0} u = \left(\begin{array}{cc}
 u_{ww}&  u_{wx_0}  \\
 u_{x_0 w}  &  u_{x_0 x_0}
\end{array} \right)\in S(2)
$$
its second derivatives with respect to the first two components $(w,x_0)$,
whenever they exist.
\end{notation}

\bigskip
The HJB equation  associated with Problem \ref{Problem1} is
\begin{equation} \label{eq:HJB1}
(\rho+\delta) v = \mathbb H\left(w,x,D v, D^2_{wx_0} v\right),
\end{equation}
where the Hamiltonian
$\mathbb H: \mathbb R \times M_2 \times  (\mathbb R \times \cald (A^*))   \times S(2) \longrightarrow \overline\R$ \footnote{Note that we allow the Hamiltonian to take values in $\overline \R$, hence to be possibly $\pm\infty$.} is, informally speaking, defined as follows
\begin{align}\label{DEF_HAMILTONIAN_INF_RETnew}
\mathbb H(w,x,p,P)  :=
\sup_{(c,B,\theta)\in \R_+\times \R_+\times \mathbb R^n}
\<\calb(w,x,\theta,c,B),p\>_{\R\times M_2}+ \frac12 Tr \Sigma(\theta,x_0) P \Sigma^*(\theta,x_0) + U(c,B)
\end{align}
where we call $\calb(w,x,\theta,c,B)$ and $\Sigma(\theta,x_0)$,\footnote{Note that $\Sigma (\theta,x_0)$ is a $2\times 2$ matrix since the component $X_1$ of the state in \eqref{DYN_W_X_INFINITE_RETIREMENT_II} has no diffusion coefficient.}  the drift and the diffusion of \eqref{DYN_W_X_INFINITE_RETIREMENT_II}, while $U$
is the utility function in the integral
\eqref{DEF_J_INF_RET}.
To compute the Hamiltonian we separate the part depending on the controls from the other one, which can be taken out of the supremum. Hence we write:
\begin{align}\label{DEF_HAMILTONIAN_INF_RET}
\mathbb H(w,x,p,P)  :=
\mathbb H_1(w,x,p,P_{22}) + \mathbb H_{max}(x_0,p_1,P_{11},P_{12}),
\end{align}
where
\begin{equation}\label{DEF_H_1_INF_RET}
\mathbb H_1(w,x,p,P_{22}) :=  (r+\delta) w p_1+x_0 p_1 + \langle x,A^*p_2  \rangle_{M_2}+
\frac{1}{2}   \sigma_y^\top  \sigma_y x_0^2 P_{22},
\end{equation}
%\begin{align}\label{DEF_H_2_INF_RET}
%\H_2(x_0,p_1,P_{11},P_{12}; c,B,\theta)
%:= [\theta^\top(\mu-r\mathbf 1) - c-\delta B]p_1 +\frac{1}{2} \theta^\top  \sigma  \sigma^\top   \theta P_{11}  +  \theta^\top  \sigma \sigma_y x_0   P_{12}
%\end{align}
and
\begin{equation}\label{DEF_H_MAX_INF_RET}
 \mathbb H_{max}(x_0,p_1,P_{11},P_{12}) := \sup_{(c,B,\theta)\in \R_+\times \R_+ \times \mathbb R^n} \mathbb H_{cv}(x_0,p_1,P_{11},P_{12}; c,B,\theta)
\end{equation}
with
\begin{align}\label{DEF_H_cv_INF_RET}
\mathbb H_{cv}
(x_0,p_1,P_{11},P_{12}; c,B,\theta):= & \frac{ c^{1-\gamma}}{1-\gamma}
+ \frac{\delta \big(k B\big)^{1-\gamma}}{1-\gamma}
%+\H_2(x_0,p_1,P_{11},P_{12}; c,B,\theta)
+[\theta^\top(\mu-r\mathbf 1) - c-\delta B]p_1
\\
&+\frac{1}{2} \theta^\top  \sigma  \sigma^\top   \theta P_{11}  +  \theta^\top  \sigma \sigma_y x_0   P_{12}
%\\
%\nonumber
%=& [\theta^\top(\mu-r\mathbf 1) - c-\delta B]p_1
%+\frac{1}{2} \theta^\top  \sigma  \sigma^\top   \theta P_{11}  +  \theta^\top  \sigma \sigma_y x_0   P_{12}
\nonumber
\\
\nonumber
:= &
\frac{ c^{1-\gamma}}{1-\gamma}- c p_1
+ \frac{\delta \big(k B\big)^{1-\gamma}}{1-\gamma} -\delta B p_1    \\
\nonumber
&+\theta^\top(\mu-r\mathbf 1)p_1
+\frac{1}{2} \theta^\top  \sigma  \sigma^\top   \theta P_{11}  +  \theta^\top  \sigma \sigma_y x_0   P_{12}.
\end{align}

%Note that, due to the dynamics of $X$, the Hamiltonian will not depend on the second derivatives
%of the value function with respect to $x_1$.
%The term $\mathbb H_{cv}$ can be rewritten as
%\begin{small}
%\begin{align*}
%\mathbb H_{cv}&(w, x,p,P; c,B,\theta)\\
%=
%&
%\frac{ c^{1-\gamma}}{1-\gamma}- c p_1
%+ \frac{\delta \big(k B\big)^{1-\gamma}}{1-\gamma} -\delta B p_1    \\
%&+\frac{P_{11}}{2}
%\left[  \theta  + (\sigma\sigma^\top)^{-1}  \left\{
%(\mu-r \mathbf 1) \frac{p_1}{P_{11}}+ \sigma \sigma_y x_0\frac{P_{12}}{P_{11}}
%\right\}  \right]^\top  (\sigma\sigma^\top)
%\left[  \theta  + (\sigma\sigma^\top)^{-1}  \left\{
%(\mu-r \mathbf 1) \frac{p_1}{P_{11}}+ \sigma \sigma_y  x_0 \frac{P_{12}}{P_{11}}
%\right\}  \right]\\
%&-\frac{1}{2P_{11}}
%\left[(\mu-r \mathbf 1)  p_1+ \sigma \sigma_y x_0 P_{12}\right]^\top
%(\sigma\sigma^\top)^{-1}
%\left[(\mu-r \mathbf 1) p_1+ \sigma \sigma_y  x_0 P_{12}\right].
%\end{align*}
%\end{small}
Now note that, thanks to the last equality above, whenever $p_1>0$ and $P_{11}<0$,
the maximum in (\ref{DEF_H_MAX_INF_RET}) is achieved at
\begin{eqnarray}\label{MAX_POINTS_HAMILTONIAN__INF_RET}
\left\{
\begin{split}
c^{*} &:=  %M^{ -b}
p_1^{-\frac{1}{\gamma}}  {,}  \\
B^{*}&:=k^{ -b} p_1^{-\frac{1}{\gamma}}   {,}  \\
\theta^{*} &:=   - (\sigma\sigma^\top)^{-1} \frac{(\mu-r\mathbf 1) p_1 + \sigma  \sigma_y x_0  P_{12}}
{P_{11}},
\end{split}
\right.
\end{eqnarray}
where
\begin{equation}\label{DEB_b}
b= 1-\frac{1}{\gamma}.
\end{equation}
Hence, for $p_1>0$ and $P_{11}<0$ we have, by simple computations,
\begin{eqnarray}\label{EQ_EXPL_HAMILTONIAN_INF_RET}
\begin{split}
\mathbb H(w,x,p,P)=
&  (r+\delta) w p_1+x_0 p_1+\langle x,A^*p_2  \rangle_{M_2}
+\frac{\gamma}{1-\gamma}  p_1^{b} \big(%M^{-b}
1+\delta k^{-b}  \big)\\
&+\frac{1}{2}   \sigma_y^\top  \sigma_y x_0^2 P_{22}\\
&-\frac{1}{2P_{11}}
\left[(\mu-r \mathbf 1)  p_1+ \sigma \sigma_y x_0 P_{12}\right]^\top
(\sigma\sigma^\top)^{-1}
\left[(\mu-r \mathbf 1) p_1+ \sigma \sigma_y  x_0 P_{12}\right].
\end{split}
\end{eqnarray}
%By Proposition \ref{CONCAVITY_VALUE_FUNCTION_INF_RET},
%the value function is concave in $w$,
%thus
%$$   V_{ww}\leq 0,$$
Therefore, if the unknown $v$ satisfies $v_w>0$ and $v_{ww}< 0$,
the HJB equation in (\ref{eq:HJB1}) reads
\begin{eqnarray}\label{HJB_CLEAN_INF_RET}
\begin{split}
(\rho+\delta) v=
&  (r+\delta) w v_w+x_0  v_w +\langle x,A^* v_x  \rangle_{M_2}+
\frac{\gamma}{1-\gamma}  v_w^{b} \big(1 %M^{-b}
+\delta k^{-b}  \big)\\
&+\frac{1}{2}   \sigma_y^\top  \sigma_y x_0^2 v_{x_0x_0}\\
&-\frac{1}{2v_{ww}}
\left[(\mu-r \mathbf 1)  v_w+ \sigma \sigma_y x_0 v_{w x_0}\right]^\top
(\sigma\sigma^\top)^{-1}
\left[(\mu-r \mathbf 1) v_w+ \sigma \sigma_y  x_0 v_{w x_0}\right].
\end{split}
\end{eqnarray}
On the other hand we must also note that, when $p_1<0$ or $P_{11}>0$,
the Hamiltonian $\mathbb H$ is $+\infty$, while, when $p_1P_{11}=0$,
different cases may arise depending on $\gamma$ and on the sign of other terms.

%IS NEEDED wHAT FOLLOwS??? CHECK INITIAL TIME OF x
%%
% Let us denote by $w^{w,x}(s;c,B, \theta)$ the value of the process $w$ in $s$,
% starting at time $0$ in $ w$, with $x(0)=x$ and following the strategy $(c,B, \theta)$,
% and with $x^x (s)$ the value of the process $x$ in $s$,
% starting at time $t$ in $x$.
%%
%
\begin{definition}\label{DEF_SUPERSOLUTION_INF_RET}
A function $ u : \calh_{++} \longrightarrow \mathbb R $
is a \emph{classical solution}
% (respectively \emph{classical supersolution})
of the HJB equation (\ref{eq:HJB1}) in $\calh_{++}$
if the following holds:
\begin{itemize}
  \item[(i)] $u$ is continuously Fr\'{e}chet differentiable in $\calh_{++}$ and admits
      continuous second derivatives with respect to $(w,x_0)$ in $\calh_{++}$;
  \item[(ii)] $u_x(w,x) \in \cald(A^*)$ for every $(w,x)\in \calh_{++}$
  and $A^*u_x$ is continuous in $\calh_{++}$;
  \item[(iii)]
for all $(w,x) \in \calh_{++} $ we have
\begin{equation}\label{EQ_SOLUTION_INF_RET}
\begin{split}
(\rho + \delta) u -
\mathbb H \big(w,x,Du, D^2_{wx_0} u\big)= 0.
\end{split}
\end{equation}
%(respectively
%\begin{eqnarray}\label{EQ_SUPERSOLUTION_INF_RET}
%\begin{split}
%(\rho + \delta) u -
%\mathbb H \big(w,x,Du, D^2_{wx_0} u\big)
% \geq 0.)
%\end{split}
%\end{eqnarray}
\end{itemize}
\end{definition}

\begin{remark}\label{REMDEF_SUPERSOLUTION_INF_RET}
It is important to note that, in the Hamiltonian $\mathbb H_1$,
the natural infinite dimensional term involving $A$ should be written as
$\<Ax,p_2\>_{M_2}$, which would make sense only for $x \in \cald(A)\cap  \calh_{++}$.
Here we decided to write it in the different form $\<x,A^*p_2\>_{M_2}$,
which makes sense for all $x\in \calh_{++}$ but only if $p_2$ (which stands for $u_{x}$)
belongs to $\cald(A^*)$. This choice substantially means that,
on classical solutions, we assume a further regularity
on the gradient. This will simplify the application of Ito's formula in next proposition and will be enough for our needs, as the explicit
solution that we compute satisfies such a regularity.
\end{remark}

%\vspace{0,5truecm}

%\blu{
\begin{remark}\label{REMzerosol}
In Definition \ref{DEF_SUPERSOLUTION_INF_RET} above (classical solution $u$ of the HJB equation \eqref{eq:HJB1})
we did not include the requirement $u_w>0$ and $u_{ww}< 0$
even if it seems natural in this context.
A reason is that, in the case $\gamma>1$,
the function $u\equiv 0$ is indeed a solution
(see Subsection \ref{SSE:GAMMA>1}).
\end{remark}

\begin{proposition}\label{PROP_COMPARISON_FINITENESS_VAL_FUN_INF_RET}
%Assume Hypothesis \ref{HYP_POSITIVITY_NU_INF_RET}.
%\begin{itemize}
%\item[]
%\item[(i)]
Define, for $(w,x) \in \calh_{++}$,
\begin{equation}\label{EQ_GUESS_BAR V_INF_RET}
\bar v (w,x):=\frac{f_{\infty}^{\gamma}
\Gamma_\infty^{1-\gamma}(w,x)}{1-\gamma},
\end{equation}
with $f_{\infty}>0$
defined as
\begin{equation}\label{EXPRESSIONS_f_INF_RET}
 f_{\infty}:= (1 + \delta k^{-b}) \nu,
 \end{equation}
where
\begin{equation}\label{DEF_nu}
\nu  := \frac{\gamma}{\rho + \delta -(1-\gamma) (r + \delta +\frac{\kappa^\top \kappa}{2\gamma })}>0,
\end{equation}
$\Gamma_{\infty}$ defined in Notation \ref{not1},
and $b$ as in (\ref{DEB_b}).
Then, $\bar v$ is a classical solution of the HJB equation
(\ref{eq:HJB1}) on $\calh_{++}$.
%
%\item[(ii)]
%The value function $V$ is finite in $\calh_{++}$ and satisfies, still in $\calh_{++}$:
%$$V\leq \bar v.$$
%%Let a positive function $ u $
%%be a classical supersolution of equation (\ref{FORMAL_HJB_INF_RET})
%%such that, for some $C>0$ and $N\ge 0$,
%%\begin{equation}\label{eq:polgrowthsupersol}
%%|u(w,x)|+|Du(w,x)|+|A^*u_x(w,x)|+|D^2_{wx_0}u(w,x)|\le C(1+|(w,x)|^N_{\calh}),
%%\quad \forall (w,x)\in \calh_{++}
%%\end{equation}
%%Let $V$ be the value function defined in (\ref{DEF_VALUE_FUNCTION_INF_RET}).
%%Then, on $\calh_+$,
%%$$V\leq u.$$
%\end{itemize}
\end{proposition}

\begin{proof}
Let $\bar v $ as in (\ref{EQ_GUESS_BAR V_INF_RET}).
Since $\Gamma_\infty$ is linear, it is immediate to check that $\bar v$ satisfies the assumption required at point (i) of Definition \ref{DEF_SUPERSOLUTION_INF_RET}.
Moreover we have, in $\calh_{++}$,
					 \begin{align}\label{DERIVATIVES_BAR_V_INF_RET}
						\begin{split}
							\bar v_w(w,x)&=  f_{\infty}^{\gamma}
\Gamma_{\infty}^{-\gamma}(w,x),
\\
\bar v_x(w,x)&= f_{\infty}^{\gamma}\Gamma_{\infty}^{-\gamma}(w,x) (g_\infty{,}h_\infty),\\
%							\bar v_{x_0}(w,x)&= f_{\infty}^{\gamma}\Gamma_{\infty}^{-\gamma}(w,x)g_{\infty},\\
%							\bar v_{x_1}(w,x)&= f_{\infty}^{\gamma}\Gamma_{\infty}^{-\gamma}(w,x) h_{\infty},\\
							\bar v_{w w}(w,x)&= - \gamma f_{\infty}^{\gamma}\Gamma_{\infty}^{-\gamma-1}(w,x),\\
							\bar v_{w x_0}(w,x)&=   - \gamma f_{\infty}^{\gamma}\Gamma_{\infty}^{-\gamma-1}(w,x) g_{\infty},\\
							\bar v_{x_0 x_0}(w,x)&= - \gamma f_{\infty}^{\gamma}\Gamma_{\infty}^{-\gamma-1}(w,x)  g_{\infty}^2.
						\end{split}
					\end{align}
Hence, since $(g_\infty,h_\infty)\in \cald (A^*)$ by
Lemma \ref{LEMMA_DIFF_EQ_G_INFTY_H_INFTY}, then also point (ii)
of Definition \ref{DEF_SUPERSOLUTION_INF_RET} holds for $\bar v$.
Concerning point (iii) of Definition \ref{DEF_SUPERSOLUTION_INF_RET}
we observe first that $\bar v_w>0$ and $\bar v_{ww}<0$ on $\calh_{++}$
and then we can use the explicit expression of the HJB equation
\eqref{HJB_CLEAN_INF_RET} which holds in such a case.
Using the expression of $A^*$ given in Proposition \ref{pr:adjoint} we have, in $\calh_{++}$,
					\begin{align*}
						\langle x,A^*\bar v_x(w,x)  \rangle_{M_2} &= f_{\infty}^{\gamma} \Gamma_{\infty}^{-\gamma}(w,x)  \langle x, A^*(g_{\infty},  h_{\infty}) \rangle_{M_2}\\
						%&= f_{\infty}^{\gamma} \Gamma_{\infty}^{-\gamma}(w,x)  \langle x, A^*(g_{\infty},  h_{\infty}) \rangle_{M_2}\\
						& = f_{\infty}^{\gamma} \Gamma_{\infty}^{-\gamma} (w,x)\left[ x_0\mu_y g_{\infty}+ x_0  h_{\infty}(0) +\langle x_1, -  h^{\prime}_{\infty} +g_{\infty}\phi\rangle\right]{.}
					\end{align*}
Moreover we have, omitting the variables $(w,x)$ for simplicity of notation,
$$
\frac{\gamma}{1-\gamma}  \bar v_w^{b} \big(1+\delta k^{-b}  \big)
=
 \frac{\gamma}{1-\gamma} f_{\infty}^{\gamma-1} \Gamma_{\infty}^{1-\gamma}
(1+\delta k^{-b}  )
$$
$$
\frac{1}{2} \sigma_y^\top  \sigma_y x_0^2 \bar v_{x_0x_0}
=
-\frac{\gamma}{2} \sigma_y^\top  \sigma_y x_0^2
f_{\infty}^{\gamma}\Gamma_{\infty}^{-\gamma-1}g_{\infty}^2
$$
and, using, in the last line below, the definition of $\kappa$ in \eqref{DEF_KAPPA},
\begin{align*}
&- \frac{1}{2\bar v_{ww}}
\left[(\mu-r \mathbf 1) \bar  v_w+ \sigma \sigma_y x_0 \bar v_{w x_0}\right]^\top
(\sigma\sigma^\top)^{-1}
\left[(\mu-r \mathbf 1) \bar v_w+ \sigma \sigma_y  x_0 \bar v_{w x_0}\right]
=
\\
&=
\frac{1}{2\gamma}f_{\infty}^{-\gamma}\Gamma_{\infty}^{1+\gamma}
\left[
(\mu -r {\bf 1}) f_{\infty}^{\gamma}\Gamma_{\infty}^{-\gamma}
-\gamma x_0 f_{\infty}^{\gamma}\Gamma_{\infty}^{-1-\gamma}g_\infty \sigma \sigma_y
\right]^\top
(\sigma \sigma^\top)^{-1}
\left[
(\mu -r {\bf 1}) f_{\infty}^{\gamma}\Gamma_{\infty}^{-\gamma}
-\gamma x_0 f_{\infty}^{\gamma}\Gamma_{\infty}^{-1-\gamma}g_\infty \sigma \sigma_y
\right]=
\\
&=
\frac{1}{2\gamma}f_{\infty}^{\gamma}\Gamma_{\infty}^{1-\gamma}
\left[
\mu -r {\bf 1} -\gamma x_0 \Gamma_{\infty}^{-1}g_\infty \sigma \sigma_y
\right]^\top
(\sigma \sigma^\top)^{-1}
\left[
\mu -r {\bf 1} -\gamma x_0 \Gamma_{\infty}^{-1}g_\infty \sigma \sigma_y
\right]
=
\\
&=
\frac{1}{2\gamma}f_{\infty}^{\gamma}\Gamma_{\infty}^{1-\gamma}
\left[\kappa^\top \kappa- \gamma x_0 \Gamma_{\infty}^{-1}g_\infty
\left(
%(\mu -r {\bf 1})^\top (\sigma \sigma^\top)^{-1} \sigma \sigma_y
%+(\sigma \sigma_y)^\top
%(\sigma \sigma^\top)^{-1} (\mu -r {\bf 1})
\kappa^\top \sigma_y + \sigma_y^\top \kappa
\right)
+ \gamma^2 x_0^2 \Gamma_{\infty}^{-2}g_\infty^2
%(\sigma \sigma_y)^\top (\sigma \sigma^\top)^{-1}(\sigma \sigma_y)
\sigma_y^\top \sigma_y
\right].
\end{align*}
Now we substitute $\bar v$ inside \eqref{HJB_CLEAN_INF_RET}
using the last three equalities. Then we multiply by $f_{\infty}^{-\gamma}\Gamma_{\infty}^{\gamma}(w,x)$,
obtaining, in $\calh_{++}$,
%write equation (\ref{HJB_bar_v_INF_RET}) as
\begin{align}%\label{HJB_INF_RET}
\begin{split}
\frac{(\rho +\delta)}{ 1-\gamma}\Gamma_{\infty} (w,x)
=
(r+\delta) w +  x_0
+x_0 \mu_y g_{\infty}  + x_0  h_{\infty}(0)
+\< x_1,-  h^{\prime}_{\infty}+ g_{\infty}\phi\>
\\
+ \frac{\gamma}{1-\gamma} f_{\infty}^{-1} \Gamma_{\infty}(w,x)
(1+\delta k^{-b}  )
 +\frac{\kappa^\top \kappa}{2\gamma }\Gamma_{\infty}  (w,x) -\kappa^\top\sigma_y x_0 g_{\infty}.
						\end{split}
					\end{align}
%with $\kappa$ as in (\ref{DEF_KAPPA})
					Recalling (\ref{EQ_g_h_infty})
					%in Lemma \ref{LEMMA_DIFF_EQ_G_INFTY_H_INFTY}
					the equality above can be rewritten as
					\begin{align}
						\begin{split}
\frac{(\rho +\delta)}{ 1-\gamma}\Gamma_{\infty}(w,x)
=
(r+\delta) w +x_0 +x_0 \mu_y g_{\infty}  + x_0 \big(   \beta g_{\infty} -1\big) +(r+\delta)\langle x_1,h_{\infty} \rangle
\\
+\frac{\gamma}{1-\gamma} f_{\infty}^{-1} \Gamma_{\infty}(w,x)(1
							+\delta k^{-b}  )
+\frac{\kappa^\top \kappa}{2\gamma }\Gamma_{\infty}  (w,x) -\kappa^\top\sigma_y x_0  g_{\infty},
						\end{split}
\end{align}
which, by the definitions of $\beta$ in \eqref{DEF_BETA} and of $\Gamma_{\infty}$
in \eqref{eq:defGamma}, reads
\begin{align*}
\begin{split}
\frac{(\rho +\delta)}{ 1-\gamma}\Gamma_{\infty}(w,x)
=
(r+\delta) \Gamma_{\infty}(w,x)+ \frac{\gamma}{1-\gamma} f_{\infty}^{-1} \Gamma_{\infty}(w,x)(1 							 +\delta k^{-b}  )
+ \frac{\kappa^\top \kappa}{2\gamma }\Gamma_{\infty}(w,x) .
						\end{split}
					\end{align*}
We can now cancel $\Gamma_\infty(w,x)$, since it is strictly positive,
and use the definition of $f_\infty$ in \eqref{EXPRESSIONS_f_INF_RET},
getting
\begin{align*}
\begin{split}
\frac{(\rho +\delta)}{ 1-\gamma}
- (r+\delta) -\frac{\gamma}{1-\gamma}\nu^{-1}
-\frac{\kappa^\top \kappa}{2\gamma }=0.
						\end{split}
					\end{align*}
This holds true thanks to the definition of $\nu$ given in (\ref{DEF_nu}),
so the claim is proved.
\end{proof}

\begin{remark}\label{rm:vonboundary}
The function $\bar v$ can be defined also in $\calh_+$ by setting,
on $\partial\calh_+=\{\Gamma_\infty(w,x)=0\}$,
$$\bar v(w,x)=0, \quad \hbox{when $\gamma\in (0,1)$},
\qquad and \qquad
\bar v(w,x)=-\infty, \quad \hbox{when $\gamma\in (1,+\infty)$}.
$$
From now on we will consider $\bar v$ defined on $\calh_+$.
\end{remark}

In next subsection we use the function $\bar v$ to prove the fundamental identity, the key step to get the Verification Theorem and the existence of optimal feedbacks, see
Theorems \ref{th:VERIFICATION_THEOREM_INF_RET} and \ref{th:VERIFICATIONgamma>1}.

\subsection{A lemma on the evolution of $\Gamma_\infty$ on the admissible paths}
\label{SSE:LEMMAGAMMA}

We need the following lemma.

\begin{lemma}\label{lm:Gammabar}
%Assume Hypothesis \ref{HYP_POSITIVITY_NU_INF_RET}.
%Let $\bar v$ be the function defined in \eqref{EQ_GUESS_BAR V_INF_RET}.
Fix $\left(w,x\right) \in \calh_{+}$ and let  $\pi=\left(c,B,\theta\right) \in \Pi\left(w,x\right)$. Then we have the following:
\begin{itemize}
  \item[(i)] The process
$$
\bar\Gamma_\infty (t):=\Gamma_\infty (W^{w,x}(t;c,B, \theta),X^{x} (t))
$$
satisfies the equation
	\begin{align}\label{eq:GammaProcessSDE}
	\begin{split}
	d\bar\Gamma_{\infty} (t)
	=
	&\left[(r+\delta)\bar\Gamma_{\infty} (t)    - c(t) - \delta B(t)
	+ \Big(  \theta^\top (t) \sigma +  g_{\infty}  X_0(t)  \sigma_y^\top \Big)\kappa\right] dt\\
	&+ \left[\theta^\top(t) \sigma +  g_{\infty}   X_0(t)  \sigma_y^\top \right]dZ_t,
	\end{split}
	\end{align}
  \item[(ii)]
Assume that $\Gamma_\infty(w,x)=0$, i.e. that $\bar \Gamma_\infty(0)=0$.
Then for every $t\ge 0$ it must be
$$
\bar \Gamma_\infty(t)=0,  \qquad \P-a.s.
$$
and
\begin{equation}\label{eq:zerostrategy}
c(t,\omega)=0, \quad B(t,\omega)=0,
\quad
\theta^\top(t) \sigma +  g_{\infty}   X_0(t)  \sigma_y^\top =0,
\qquad
dt\otimes \P-a.e. \; in \; [0,+\infty) \times \Omega
\end{equation}
Let now $\Gamma_\infty(w,x)>0$ and set $\tau$ be the first exit time of $\left(W_\pi(\cdot),X(\cdot)\right)$ from $\calh_{++}$.
Then for every $t\ge 0$
$$
\1_{\{\tau<t\}}\bar \Gamma_\infty(t)=0,  \qquad \P-a.s.
$$
and
$$
\1_{\{\tau<t\}}(\omega) c(t,\omega)=0, \quad \1_{\{\tau<t\}}(\omega)B(t,\omega)=0,
\quad
\1_{\{\tau<t\}}\left[\theta^\top(t) \sigma +  g_{\infty}   X_0(t)  \sigma_y^\top\right] =0,
$$
$dt\otimes \P-a.e. \; in \; [0,+\infty) \times \Omega$.
%  \item[(iii)]
%what happens for strategies inside and arriving at the boundary
\end{itemize}
\end{lemma}

\begin{proof}
Fix $\left(w,x\right) \in \calh_{+}$
and $\pi=\left(c,B,\theta\right) \in \Pi\left(w,x\right) $.
Recall that $W^{w,x}(s;c,B, \theta)$ is the solution at time $s$ of
the first equation in (\ref{DYN_W_X_INFINITE_RETIREMENT_II})
starting at time $0$ in $\left(w,x\right)$
and following the strategy $\left(c,B,\theta \right)$; recall also that
$X^{x} (s)$ is the solution at time $s$ of the second equation in (\ref{DYN_W_X_INFINITE_RETIREMENT_II}), starting at time $0$ in $x$.
For readability we will use the shorthand notation
		\begin{align*}
			 W_{\pi}(s) &:=  W^{w,x}(s;c,B,\theta),\\
			 X(s ) &:=  X^{x} (s),
		\end{align*}
omitting the dependence on the controls and on the initial conditions.

Now we prove each step separately.

\medskip

{\bf Proof of (i).}
We apply Ito's formula of \cite[Proposition~1.165]{FABBRI_GOZZI_SWIECH_BOOK},
to the process
$\langle (g_{\infty},h_{\infty}),(X_0(t),X_1(t))\rangle{_{M_2}}$
getting
\begin{align}\label{eq:hinftyX1}
\begin{split}
d\langle (g_{\infty},h_{\infty}),& (X_0(t),X_1(t))\rangle_{M_2}
=
\\
=
&\langle A^*(g_{\infty},h_{\infty}), (X_0(t),X_1(t)) \rangle_{M_2} dt
+\langle (g_{\infty},h_{\infty}), \big(C(X_0(t),X_1(t))\big)_0^{\top} dZ_t \rangle_{M_2} \\[1mm]
=&  \big[  \mu_y  g_{\infty} X_0(t)+ h_{\infty}(0) X_0(t) +
\<-h_{\infty}^{\prime}+  g_{\infty} \phi,X_1(t)\>\big]  dt
+ g_{\infty}  X_0(t)  \sigma_y^\top dZ_t
\\[1mm]
=& \big[\big(\mu_y + \beta\big) g_{\infty} X_0(t) -X_0(t)+ (r+\delta)\langle h_{\infty},   X_1(t) \rangle
\big]  dt + g_{\infty}  X_0(t)  \sigma_y^\top dZ_t,
\end{split}
\end{align}
where the last equality follows from (\ref{EQ_g_h_infty}).
Therefore, writing
$$
\bar\Gamma_{\infty}(t)=\Gamma_{\infty} \big( W_\pi(t),X(t)  \big)
= W_\pi(t)+ g_{\infty}X_0(t) + \< h_{\infty}, X_1(t)\>{,} $$
%From (\ref{DYN_W_X_INFINITE_RETIREMENT_II})
%(\ref{DIFFERENTIAL_G_X_0_H_X_1_INF_RET})
we have, from (\ref{DYN_W_X_INFINITE_RETIREMENT_II}) and \eqref{eq:hinftyX1},
	\begin{align}\label{eq:bargammaito}
	\begin{split}
	d\bar\Gamma_{\infty} (t)
	=
	&\left[W(t) (r+\delta) + \theta^\top(t) (\mu-r\mathbf 1)  +  X_0(t) - c(t) - \delta B(t) \right.\\
	&+ \left.\big(\mu_y + \beta\big) g_{\infty} X_0(t) -X_0(t)+ (r+\delta)\langle h_{\infty},   X_1(t) \rangle
	\right] dt\\
	&+ \left[ \theta^\top(t)  \sigma  +  g_{\infty}    X_0(t)  \sigma_y^\top \right]dZ_t\\
	=
	&\left[(r+\delta)\bar\Gamma_{\infty} (t)    - c(t) - \delta B(t)
	+ \Big(  \theta^\top (t) \sigma +  g_{\infty}  X_0(t)  \sigma_y^\top \Big)\kappa\right] dt\\
	&+ \left[\theta^\top(t) \sigma +  g_{\infty}   X_0(t)  \sigma_y^\top \right]dZ_t,
	\end{split}
	\end{align}
	where we have used $\mu_y + \beta = r+\delta +  \sigma_y^\top\kappa$ (see \eqref{DEF_BETA}). This is the claim.

\medskip

{\bf Proof of (ii).}
%Assume that $\Gamma_\infty(w,x)=0$ i.e. that $\bar \Gamma_\infty (0)=0$.
%It is clear that the strategy given by \eqref{eq:zerostrategy} is admissible.
%On the other hand
%Let $\pi\in \Pi(w,x)$.
%be an admissible strategy.
By Girsanov's Theorem (see e.g. \cite[Theorem  3.5.1]{KARATZAS_SHREVE_91}, for any $T>0$, under the probability (depending on $T$, defined on $\calf_T$ and, there, equivalent to $\P$)
\begin{equation}\label{eq:Ptdef}
\tilde \P_T= \exp\left(-\kappa^\top Z(T)- \frac12\kappa^\top \!\kappa\, T\right)\P,
\end{equation}
the process $t \mapsto \tilde Z(t)= \kappa t + Z(t)$
is a $d$-dimensional Brownian motion on $[0,T]$ and $\bar\Gamma_\infty$
satisfies, on $[0,T]$,
$$
d\bar\Gamma_{\infty} (t)=
\left[(r+\delta)\bar\Gamma_{\infty} (t)    - c(t) - \delta B(t)\right] dt
+ \left[\theta^\top(t) \sigma +  g_{\infty}   X_0(t)  \sigma_y^\top \right]
d\tilde Z_t.
$$
Hence we obtain, under $\tilde\P_T$, for every
%stopping times
$0\le t\le T$,
\begin{equation}\label{eq:Gammatau}
\bar\Gamma_\infty(t)=e^{(r+\delta)t}
\left[\bar\Gamma_\infty(0)-\int_{0}^{t} e^{-(r+\delta)s}(c(s) + \delta B(s))ds
+  \int_{0}^{t} e^{-(r+\delta)s}\left[\theta^\top(s) \sigma +  g_{\infty}   X_0(s)  \sigma_y^\top \right]d\tilde Z_s\right].
\end{equation}
Setting $\Gamma_\infty^0(t):=e^{-(r+\delta)t}\bar\Gamma_\infty(t)$
the above \eqref{eq:Gammatau} is rewritten as
\begin{equation}\label{eq:Gamma0tau}
\Gamma^0_\infty(t)=
\Gamma^0_\infty(0)-\int_{0}^{t} e^{-(r+\delta)s}(c(s) + \delta B(s))ds
+  \int_{0}^{t} e^{-(r+\delta)s}\left[\theta^\top(s) \sigma +  g_{\infty}   X_0(s)  \sigma_y^\top \right]d\tilde Z_s{,}
\end{equation}
which implies that the process $\Gamma^0_\infty(t)$ is a supermartingale under $\tilde \P_T$ on $[0,T]$.
By the optional sampling theorem we then have, for every couple of stopping times
$0\le \tau_1\le \tau_2\le T$, and calling $\tilde\E_T$
the expectation under $\tilde\P_T$,
\begin{equation}\label{eq:OST}
\tilde \E_T \left[\Gamma^0_\infty(\tau_2)| \calf_{\tau_1}\right]
\le
\Gamma^0_\infty(\tau_1), \qquad \tilde \P_T-a.s.
\end{equation}
The admissibility of the strategy $\pi$, and the fact that $\P$ and $\tilde\P_T$ are equivalent on $\calf_T$, implies that
$\Gamma^0_\infty(\tau_2)\ge 0$, $\tilde \P_T$-a.s., hence also
$$
%\1_{\{\tau<T\}}
\tilde \E_T \left[\Gamma^0_\infty(\tau_2)| \calf_{\tau_1}\right]
%=
%\tilde \E_T \left[\Gamma^0_\infty(\tau_2)\1_{\{\tau<T\}}| \calf_{\tau_1}\right]
\ge 0,\qquad \tilde \P_T-a.s.
$$
Now let $\tau_1:=\tau \wedge T$ where $\tau$ is the first exit time
of $(W_\pi(t),X(t))$ from $\calh_{++}$ (which is taken to be identically $0$
when $\Gamma_\infty(w,x)=0$). Hence $\Gamma^0_\infty(\tau_1)=0$ on $\{\tau<T\}$ and, from \eqref{eq:OST}, we get
$$
\1_{\{\tau<T\}}
\tilde \E_T \left[\Gamma^0_\infty(\tau_2)| \calf_{\tau_1}\right]
=
\tilde \E_T \left[\Gamma^0_\infty(\tau_2)\1_{\{\tau<T\}}| \calf_{\tau_1}\right]
=0, \qquad \tilde\P_T-a.s.
$$
and, consequently,
\begin{equation}\label{eq:GammaInd0}
\Gamma^0_\infty(\tau_2)\1_{\{\tau<T\}}=0, \qquad \tilde\P_T-a.s.
\end{equation}
We now use \eqref{eq:Gamma0tau} to compute
$\Gamma^0_\infty(\tau_2)-\Gamma^0_\infty(\tau_1)$ getting
\begin{equation}\label{eq:Gamma0taubis}
\Gamma^0_\infty(\tau_2)-\Gamma^0_\infty(\tau_1)=
-\int_{\tau_1}^{\tau_2} e^{-(r+\delta)s}(c(s) + \delta B(s))ds
+  \int_{\tau_1}^{\tau_2} e^{-(r+\delta)s}
\left[\theta^\top(s)\sigma+g_{\infty}X_0(s)\sigma_y^\top\right]d\tilde Z_s.
\end{equation}
Again using the optional sampling theorem we get
$$
\tilde \E_T \left[\Gamma^0_\infty(\tau_2)| \calf_{\tau_1}\right]
-\Gamma^0_\infty(\tau_1)=
-\tilde \E_T \left[\int_{\tau_1}^{\tau_2} e^{-(r+\delta)s}(c(s) + \delta B(s))ds
| \calf_{\tau_1}\right], \qquad \tilde\P_T-a.s.
$$
Hence, taking $\tau_2\equiv T$
$$
0\le  \1_{\{\tau<T\}}
\tilde\E_T \left[\Gamma^0_\infty(T)| \calf_{\tau_1}\right]
=
-\tilde \E_T \left[\int_{0}^{T}  \1_{\{\tau<s\}} e^{-(r+\delta)s}(c(s) + \delta B(s))ds
| \calf_{\tau_1}\right], \qquad \tilde\P_T-a.s.
$$
which implies
\begin{equation}\label{eq:cB0PT}
\1_{\{\tau<s\}}(\omega) c(s,\omega)= \1_{\{\tau<s\}}(\omega)B(s,\omega)=0,
\qquad ds \otimes \tilde\P_T-a.e.\; in \;[0,T]\times \Omega.
\end{equation}
%by admissibility of $(c,B,\theta)$
% taking the expectation $\tilde\E_T$ in \eqref{eq:OST}, we get
%$$
%\tilde \E_T\left[\tilde \E_T \left[\Gamma^0_\infty(\tau_2)| \calf_{\tau_1}\right]
%\1_{\{\tau<T\}} \right]
%\le
%$$
%
%$$
%\tilde\E[\bar\Gamma_\infty(t)]=e^{(r+\delta)t}
%\left[\bar\Gamma_\infty(0)-\tilde \E\int_{0}^{t} e^{-(r+\delta)s}(c(s) + \delta B(s))ds
%\right]
%+  \int_0^t e^{-(r+\delta)s}\left[\theta^\top(t) \sigma +  g_{\infty}   X_0(t)  \sigma_y^\top \right]d\tilde Z_t\right].
%$$
%Assume now that $\bar\Gamma_\infty(0)=0$.
%Since $\P$ and $\tilde \P$ are equivalent, by admissibility of $(c,B,\theta)$
%we get that $\tilde\E[\bar\Gamma_\infty(t)]\ge 0$, which then implies
%$$
%\tilde \E\int_{0}^{t} e^{-(r+\delta)s}(c(s) + \delta B(s))ds=0
%$$
%By the positivity of $c(\cdot)$ and $B(\cdot)$ and by
%the arbitrariness of $t$ and $T$ we immediately get that
%\begin{equation}\label{eq:zerostrategyproof1}
%c(t,\omega)=0, \quad B(t,\omega)=0,
%%\quad
%%\theta^\top(t) \sigma +  g_{\infty}   X_0(t)  \sigma_y^\top =0,
%\qquad
%dt\otimes \P-a.e. \; in \; [0,+\infty) \times \Omega
%\end{equation}
We now multiply \eqref{eq:Gamma0taubis} by $\1_{\{\tau<T\}}$
and we use \eqref{eq:GammaInd0} and \eqref{eq:cB0PT} to get
\begin{equation}\label{eq:Gamma0tauter}
0
%=\1_{\{\tau<T\}}\Gamma^0_\infty(\tau_2)
=
\int_{0}^{\tau_2} e^{-(r+\delta)s}\1_{\{\tau<s\}}
\left[\theta^\top(s)\sigma+g_{\infty}X_0(s)\sigma_y^\top\right]d\tilde Z_s, \qquad \tilde\P_T-a.s.
%\bar\Gamma_\infty(t)=e^{(r+\delta)t}
%\left[\int_{0}^{t} e^{-(r+\delta)s}
%\left[\theta^\top(s) \sigma +g_{\infty}X_0(s)\sigma_y^\top \right]
%d\tilde Z_s\right].
\end{equation}
%By admissibility the left hand side is nonegative $\tilde\P_T$-a.s.
Since the integral of the right hand side is a martingale the above implies
that
\begin{equation}\label{theta0PT}
\1_{\{\tau<s\}} \theta^\top(s) \sigma +g_{\infty}X_0(s)\sigma_y^\top=0,
\qquad ds \otimes \tilde\P_T-a.e.\; in \;[0,T]\times \Omega.
\end{equation}
Using \eqref{eq:GammaInd0}-\eqref{eq:cB0PT}-\eqref{theta0PT}, that $\P$ and $\tilde \P_T$ are equivalent on $\calf_T$, and the arbitrariness of $T$, we get the claim.
%
%\smallskip
%Let now $\Gamma_\infty(w,x)>0$ and set $\tau$ be the first exit time of $\left(W_\pi(\cdot),X(\cdot)\right)$ from $\calh_{++}$.
%Then, on the set $\{\tau <+\infty\}$ we have the same as above for $t\ge\tau$.
\end{proof}

\subsection{The fundamental identity when $\gamma \in (0,1)$}
\label{SSE:FUNDIDGAMMA01}

We start with the following lemma.

\begin{lemma}\label{LEMMA_LIMIT_AT_INFTY_GUESS_VALUE_FUNCTION_INFINITE_RETIREMENT}
Let $\gamma \in (0,1)$.
For any initial condition $(w,x)\in \calh_{++}$ and for any strategy $\left(c,B, \theta\right) \in \Pi(w,x)$,
let $\tau$ be the first exit time of $\left(W_\pi(\cdot),X(\cdot)\right)$ from $\calh_{++}$. Then we have, for every $T>0$,
\begin{equation}\label{EQ_DECREASING_PROD_INF_RETnew}
\mathbb E \left[
e^{-(1-\gamma)\big(r+\delta +\frac{\kappa^\top \kappa}{2\gamma}\big)(T\wedge \tau)} %f_{\infty}^{\gamma}
%\frac{\Gamma_{\infty}^{1-\gamma} (T\wedge \tau)}{1-\gamma}
\bar v\big(W_\pi((T\wedge \tau)), X((T\wedge \tau))\big)\right] \leq
\bar v(x,w).
%f_{\infty}^{\gamma}
%		\mathbb E \left(
%\frac{\Gamma_{\infty}^{1-\gamma} (0)}{1-\gamma}\right).
\end{equation}
Moreover,\footnote{\color{black}This is a sort of transversality condition, along the lines of what is usually done in the context of the maximum principle in infinite horizon problems.}
\begin{equation}\label{EQ_DECREASING_PROD_INF_RETlim}
\lim_{T\to +\infty}
\mathbb E \left[
e^{-(\rho+\delta)(T\wedge \tau)}
\bar v\big(W_\pi((T\wedge \tau)), X((T\wedge \tau))\big)\right]
=0{.}
%\leq
%e^{-\big(\rho+\delta-(1-\gamma)(r+\delta+\frac{\kappa^\top\kappa}{2\gamma})\big)t}
% \bar v\big(w,x\big),
\end{equation}
%while, for $\gamma >1$ we have
%\begin{equation}\label{EQ_DECREASING_PROD_INF_RETlim2}
%\lim_{T\to +\infty}
%\mathbb E \left[
%e^{-(\rho+\delta)(T\wedge \tau)}
%\bar v\big(W_\pi((T\wedge \tau)), X((T\wedge \tau))\big)\right]
%=0
%\qquad if \quad \P(\tau=+\infty)=1
%\end{equation}
%\begin{equation}\label{EQ_DECREASING_PROD_INF_RETlim3}
%\lim_{T\to +\infty}
%\mathbb E \left[
%e^{-(\rho+\delta)(T\wedge \tau)}
%\bar v\big(W_\pi((T\wedge \tau)), X((T\wedge \tau))\big)\right]
%=-\infty
%\qquad if \quad \P(\tau=+\infty)<1
%\end{equation}
%%and
%%\begin{equation}\label{EQ_LIM_T_BAR_V_INF_RET}
%%\lim_{T \rightarrow +\infty} e^{-(\rho+\delta)T}\mathbb E \Big(  \bar v\big(W^{w,x}(T; c,B,\theta), X^{x}(T)\big) \Big) =0.
%%\end{equation}
\end{lemma}

\begin{proof}
Fix $\left(w,x\right) \in \calh_{+}$
and $\pi=\left(c,B,\theta\right) \in \Pi\left(w,x\right) $.
%Recall that $W^{w,x}(s;c,B, \theta)$ is the solution at time $s$ of
%the first equation in (\ref{DYN_W_X_INFINITE_RETIREMENT_II})
%starting at time $0$ in $\left(w,x\right)$
%and following the strategy $\left(c,B,\theta \right)$; recall also that
%$X^{x} (s)$ is the solution at time $s$ of the second equation in (\ref{DYN_W_X_INFINITE_RETIREMENT_II}), starting at time $0$ in $x$.
As in the previous proof, for readability, we will use the shorthand notation
		\begin{align*}
			 W_{\pi}(s) &:=  W^{w,x}(s;c,B,\theta),\\
			 X(s ) &:=  X^{x} (s),
		\end{align*}
omitting the dependence on the controls and on the initial conditions.

%If $\Gamma_\infty(w,x)=0$ then, by the previous Lemma \ref{lm:Gammabar}
%we have $\Gamma(W_\pi(t),X(t))=0$ for all $t\ge 0$, $\P$-a.s.;
%hence \eqref{EQ_DECREASING_PROD_INF_RET} is clearly true recalling Remark
%\ref{rm:vonboundary}.

Let then $\Gamma_\infty(w,x)>0$ and set $\tau$ be the first exit time of $\left(W_\pi(\cdot),X(\cdot)\right)$ from $\calh_{++}$.
We apply It\^o's formula to the process
$$
e^{-(1-\gamma)\big((r+\delta) + \frac{\kappa^\top \kappa}{ 2 \gamma }\big) t}
\bar v\big(W_\pi(t), X(t)  \big)
$$
up to time $\tau$.
Since by definition
\begin{equation}
	\bar v (w,x)=f_{\infty}^{\gamma} \frac{\Gamma_{\infty}^{1-\gamma}(w,x)}{1-\gamma},
	\end{equation}
then, using, as in the previous proof, the shorthand notation
$$
\bar\Gamma_{\infty}(t):=
\Gamma_{\infty} \big( W_\pi(t),X(t)  \big),
$$
we have
$$
e^{-(1-\gamma)\big((r+\delta) + \frac{\kappa^\top \kappa}{ 2 \gamma }\big) t}
\bar v\big(W_\pi(t), X(t)  \big)
=
e^{-(1-\gamma)\big((r+\delta) + \frac{\kappa^\top \kappa}{ 2 \gamma }\big) t} f_{\infty}^{\gamma}\frac{\bar\Gamma_{\infty}^{1-\gamma} (t)}{1-\gamma}.
$$
We first apply the finite-dimensional It\^{o}'s formula (up to time $\tau$)
to the process
$\frac{\bar\Gamma_{\infty}^{1-\gamma} (t)}{1-\gamma}$.
Since the function $h(z)=\frac{z^{1-\gamma}}{1-\gamma}$
belongs to $C^2((0,+\infty),\R)$ and not to $C^2(\R,\R)$, we use the Remark 2 to Theorem IV.3.3 of \cite{RevuzYor99} which takes account of this case.
We get, up to time $\tau$,
	\begin{align}
	\begin{split}
	d\frac{\bar\Gamma_{\infty}^{1-\gamma} (t)}{1-\gamma}	 =
	 &\bar\Gamma_{\infty}^{-\gamma}(t)\left[(r+\delta)\bar\Gamma_{\infty}(t)    - c(t) - \delta B(t)
	+\Big(\theta^\top (t) \sigma +  g_{\infty}  X_0(t)  \sigma_y^\top \Big)\kappa\right] dt\\[1mm]
	&-\frac{1}{2}\gamma \bar\Gamma_{\infty}^{-\gamma-1}
	\left[\theta^\top(t) \sigma +  g_{\infty}  X_0(t)  \sigma_y^\top  \right]^\top
	\left[\theta^\top(t) \sigma +  g_{\infty}  X_0(t)\sigma_y^\top  \right]dt\\[1mm]
	&+ \bar\Gamma_{\infty}^{-\gamma}(t)\left[\theta^\top(t) \sigma +  g_{\infty}  X_0(t)   \sigma_y^\top \right]dZ_t.
	\end{split}
	\end{align}
Hence, up to time $\tau$,
\begin{align}\label{EQ_DYN_EXP_GAMMA}
	\begin{split}
&d\left[
e^{-(1-\gamma)\left(r+\delta + \frac{\kappa^\top \kappa}{ 2 \gamma }\right)t}
f_{\infty}^{\gamma}\frac{\bar\Gamma^{1-\gamma}_\infty(t)}{1-\gamma}
%\bar v (W_\pi(t),X(t))
\right]
%f_{\infty}^{\gamma}\frac{\bar\Gamma_{\infty}^{1-\gamma} (t)}{1-\gamma}\right]
\\[2mm]
=&e^{-(1-\gamma)\left(r+\delta + \frac{\kappa^\top \kappa}{ 2 \gamma }\right) t} f_{\infty}^{\gamma}
	\bigg\{-\left(r+\delta + \frac{\kappa^\top \kappa}{ 2 \gamma }\right)
	\bar\Gamma_{\infty}^{1-\gamma} (t)\\
	&+ \bar\Gamma_{\infty}^{-\gamma}(t)\left[(r+\delta)
\bar\Gamma_{\infty}(t)- \big(c(t)+ \delta B(t) \big)
	+ \Big(  \theta^\top (t) \sigma +  g_{\infty}  X_0(t)  \sigma_y^\top \Big)\kappa\right] \\
	& -\frac{1}{2}\gamma \bar\Gamma_{\infty}^{-\gamma-1}{(t)}
	\left[\theta^\top(t) \sigma +  g_{\infty}   X_0(t) \sigma_y^\top \right]^\top
	\left[\theta^\top(t) \sigma +  g_{\infty}   X_0(t) \sigma_y^\top \right] \bigg\}dt\\
&+ e^{-(1-\gamma)\left(r+\delta + \frac{\kappa^\top \kappa}{ 2 \gamma }\right) t} f_{\infty}^{\gamma} \bar\Gamma_{\infty}^{-\gamma}(t) \left[\theta^\top(t) \sigma +  g_{\infty}  X_0(t)  \sigma_y^\top  \right]dZ_t\\[2mm]
	\end{split}
	\end{align}
\begin{align}\label{EQ_DYN_EXP_GAMMAbis}
	\begin{split}
	=&
	e^{-(1-\gamma)\left(r+\delta + \frac{\kappa^\top \kappa}{ 2 \gamma }\right) t}f_{\infty}^{\gamma}
	\Big\{ - \bar\Gamma_{\infty}^{-\gamma}(t) \big(c(t) + \delta B(t))  \\
	&- \frac{1}{2 \gamma }\bar\Gamma_{\infty}^{-\gamma-1}(t)
	\left[ \bar\Gamma_{\infty}(t)  \kappa-  \gamma  \Big(\theta^\top(t) \sigma +  g_{\infty} X_0(t)   \sigma_y^\top  \Big)       \right]^\top
	\left[ \bar\Gamma_{\infty}(t)  \kappa-  \gamma  \Big(\theta^\top(t) \sigma +  g_{\infty} X_0(t)  \sigma_y^\top   \Big)       \right]
	\Big\} dt\\
	&+ e^{-(1-\gamma)\big((r+\delta) + \frac{\kappa^\top \kappa}{ 2 \gamma }\big) t} f_{\infty}^{\gamma}\bar\Gamma_{\infty}^{-\gamma}(t) \left[\theta^\top(t) \sigma +  g_{\infty}  X_0(t)  \sigma_y^\top  \right]dZ_t.
\end{split}
	\end{align}
The strategy $(c,B,\theta) \in \Pi(w,x)$,
%hence $\Gamma_{\infty}(t) \geq 0 $, and $c,B \geq0$,
thus the drift in (\ref{EQ_DYN_EXP_GAMMAbis}) is negative.
It follows that, up to time $\tau$, the process
\begin{equation}\label{eq:supermart}
e^{-(1-\gamma)\big((r+\delta) +
\frac{\kappa^\top \kappa}{ 2 \gamma }\big) t} f_{\infty}^{\gamma}
\frac{\bar\Gamma_{\infty}^{1-\gamma} (t)}{1-\gamma}
\end{equation}
is a local $\mathbb F$-supermartingale.
%Note that it is positive, thus by Fatou's Lemma\footnote{Let $(X_t)_{t\geq 0}$ be a local positive $\mathbb F$-supermartingale, i.e., there exists an increasing sequence $(\tau_n)$ of stopping times, such that $\lim_{n \rightarrow +\infty } \tau_n = +\infty$, and for any $n\in \mathbb N$, and $s \leq t$, $\mathbb E \left(X_{t \wedge \tau_n} \mid \mathcal F_s \right) \leq  X_{s \wedge \tau_n}.$
%By Fatou's Lemma we then have $
%\mathbb E \left(   X_{t} \mid \mathcal F_s \right)
%= \mathbb E \left( \lim_n  X_{t \wedge \tau_n} \mid \mathcal F_s \right)
%\leq \lim_n  \mathbb E \left(   X_{t \wedge \tau_n} \mid \mathcal F_s \right)
%\leq  \lim_n  X_{s \wedge \tau_n}  =  X_{s }$.}
%it is an $\mathbb F$-supermartingale, up to time $\tau$.
Moreover, calling
$$
\tau_N=\inf\left\{t\ge 0: \bar\Gamma_\infty(t)\le \frac1N\right\}{,}
$$		
for $N$ sufficiently big we have $\tau_N>0$ and
both the drift and the diffusion coefficient above are integrable.
Hence $\bar \Gamma_\infty(t)$ is an integrable supermartingale up to $\tau_N$.
We then have, for every $T>0$,
\begin{eqnarray}\label{eq:supermartprimadiFatou}
\mathbb E \left[
e^{-(1-\gamma)\big(r+\delta +\frac{\kappa^\top \kappa}{2\gamma}\big)(T\wedge \tau_N)} f_{\infty}^{\gamma}
\frac{\bar\Gamma_{\infty}^{1-\gamma} (T\wedge \tau_N)}{1-\gamma}\right] \leq
f_{\infty}^{\gamma}
		\mathbb E \left(\frac{\bar\Gamma_{\infty}^{1-\gamma} (0)}{1-\gamma}\right).
		\end{eqnarray}
Since $\tau_N \nearrow \tau$, sending $N$ to $+\infty$ we get, using the definition of $\bar v$, the a.s. convergence:
$$
e^{-(1-\gamma)\big(r+\delta +\frac{\kappa^\top \kappa}{2\gamma}\big)(T\wedge \tau_N)} f_{\infty}^{\gamma}
\frac{\bar\Gamma_{\infty}^{1-\gamma} (T\wedge \tau_N)}{1-\gamma}
\to
e^{-(1-\gamma)\big(r+\delta +\frac{\kappa^\top \kappa}{2\gamma}\big)(T\wedge \tau)} \bar v (W_\pi(T\wedge \tau),X(T\wedge \tau)){.}
$$
Since $\gamma \in (0,1)$, the integrand in the left hand side of \eqref{eq:supermartprimadiFatou} is always nonegative, hence we can apply
Fatou's Lemma to get the claim (\ref{EQ_DECREASING_PROD_INF_RETnew}).

\bigskip

Now we prove the claims \eqref{EQ_DECREASING_PROD_INF_RETlim}. First, observe that, for every $T>0$,
\begin{eqnarray}
&\mathbb E \left[
e^{-(\rho+\delta)(T\wedge \tau)}
\bar v\big(W_\pi(T\wedge \tau), X(T\wedge \tau)\big)\right]
=
\nonumber
\\
&\mathbb E \left[
\mathbf{1}_{\{\tau\le T\}}e^{-(\rho+\delta)\tau}
\bar v\big(W_\pi(\tau), X(\tau)\big)\right]
+
e^{-(\rho+\delta)T}
\mathbb E \left[\mathbf{1}_{\{\tau> T\}}
\bar v\big(W_\pi(T), X(T)\big)\right]=:I_1+I_2{.}
\label{eq:decomposeindicator}
\end{eqnarray}
Second, the term $I_1$ in \eqref{eq:decomposeindicator}
is zero since, by definition of $\tau$ and continuity of $\bar \Gamma$ we have $\bar\Gamma(\tau)=0$.

Third, observe that the term $I_2$ converges to $0$. Indeed
$$
I_2=
e^{-\left[\rho+\delta -(1-\gamma)\big(r+\delta +\frac{\kappa^\top \kappa}{2\gamma}\big)
\right]T}
\mathbb E \left[\mathbf{1}_{\{\tau> T\}}
e^{-(1-\gamma)\big(r+\delta +\frac{\kappa^\top \kappa}{2\gamma}\big)T}
\bar v\big(W_\pi(T), X(T)\big)\right]
$$
$$
\le
e^{-\left[\rho+\delta -(1-\gamma)\big(r+\delta +\frac{\kappa^\top \kappa}{2\gamma}\big)\right]T}
\bar v\big(w,x\big){,}
$$
where in the last {inequality} we used (\ref{EQ_DECREASING_PROD_INF_RETnew}).
By Hypothesis \ref{HYP_BETA-BETA_INFTY}-(ii) we immediately get that
$I_2 \to 0$ as $T\to +\infty$.
The claim \eqref{EQ_DECREASING_PROD_INF_RETlim} immediately follows.
%
%Now look at the term $I_1$ in \eqref{eq:decomposeindicator}.
%If $\gamma\in (0,1)$ such term is zero since, by definition of $\tau$ and continuity of $\bar \Gamma$ we have $\bar\Gamma(\tau)=0$.
%The claim \eqref{EQ_DECREASING_PROD_INF_RETlim} immediately follows.
%
%Take now $\gamma>1$. If $\P(\tau=+\infty)=1$ then, for all $T>0$,
%$I_1=0$ as the set $\{\tau\le T\}$ is clearly empty. This gives the claim
%\eqref{EQ_DECREASING_PROD_INF_RETlim2}.
%
%If $\P(\tau=+\infty)<1$ then, for $T>0$ sufficiently big,
%$\P(\tau\le T)>0$. Since $\frac{\bar\Gamma(\tau)^{1-\sigma}}{1-\sigma}=-\infty$
%it follows that $I_1=-\infty$, which gives the claim
%\eqref{EQ_DECREASING_PROD_INF_RETlim3}.
%The claim \eqref{EQ_DECREASING_PROD_INF_RETlim2} then follows from
%Hypothesis \ref{HYP_POSITIVITY_NU_INF_RET}.
%$$
%\mathbb E \left[
%e^{-(1-\gamma)\big(r+\delta +\frac{\kappa^\top \kappa}{2\gamma}\big)(t\wedge \tau_N)} f_{\infty}^{\gamma}
%\frac{\bar\Gamma_{\infty}^{1-\gamma} (t\wedge \tau_N)}{1-\gamma}\right] \leq
%f_{\infty}^{\gamma}
%		\mathbb E \left(\frac{\bar\Gamma_{\infty}^{1-\gamma} (0)}{1-\gamma}\right).
%	$$
%		Multiplying both terms by $e^{-\big(\rho+\delta- (1-\gamma) (r+\delta +
%			\frac{\kappa^\top \kappa}{2 \gamma })  \big)t}$,
%		and recalling the definition of $\bar v$, we get (\ref{EQ_DECREASING_PROD_INF_RET}).
%		Then (\ref{EQ_LIM_T_BAR_V_INF_RET}) follows by Hypothesis \ref{HYP_POSITIVITY_NU_INF_RET}.
\end{proof}

\begin{proposition}\label{pr:fundid}
Let $\gamma \in (0,1)$. Take any initial condition $(w,x)\in \calh_{++}$ and take any admissible strategy
$\pi=\left(c,B, \theta\right) \in \Pi(w,x)$.
Let $\tau$ be the first exit time of $\left(W_\pi(\cdot),X(\cdot)\right)$ from $\calh_{++}$. Then we have, for every $T>0$,
\begin{multline}\label{EXP_VER_INF_RET_III}
 \begin{split}
&\bar v\big(w,x\big)=J(w,x;\pi)
\\[2mm]
+&\mathbb E  \int_0^{\tau} e^{-(\rho + \delta) s }\left\{
\mathbb H_{max}\big(
W_{\pi}(s),X(s),D\bar v(W_{\pi}(s),X(s)),D^2 \bar v(W_{\pi}(s),X(s))
\big)\right.
\\[2mm]
&\qquad\qquad\left.- \mathbb H_{cv}\big(
W_{\pi}(s),X(s),D\bar v(W_{\pi}(s),X(s)),D^2\bar v(W_{\pi}(s),X(s));\pi(s)
\big)
\right\} ds {.}
\end{split}
\end{multline}
\end{proposition}

\begin{proof}
Calling
$$
\tau_N=\inf\left\{t\ge 0: \bar\Gamma_\infty(t)\le \frac1N\right\}
$$		
for $N$ sufficiently big we have $\tau_N>0$. Let us take such $N$ from now on.
We want to apply It\^{o}'s formula
to the process
$$
e^{-(\rho + \delta) s}\bar v\big(W_{\pi}(s), X(s)\big)
$$
varying $s$ between $0$ and $T\wedge \tau_N$.
We were not able to find, in the current literature, an It\^o's formula which can be applied to this case. Indeed all results on It\^o's formula
in infinite dimension concern the case of functions defined over the whole space, which is not the case here.
However it is not difficult to check that the proof of Proposition 1.164 in \cite{FABBRI_GOZZI_SWIECH_BOOK} can be easily generalized to this case as our process, on $[0,\tau_N]$, lives in the set
$\left\{t\ge 0: \bar\Gamma_\infty(t)\ge \frac1N\right\}$
and, on the same set the function $\bar v$
satisfies all required assumptions.\footnote{To avoid using such a generalization of It\^o's formula it is possible to argue differently: one can use first the equation \eqref{eq:GammaProcessSDE} given in Lemma \ref{lm:Gammabar}, for the process $\bar\Gamma$ and then apply the one-dimensional It\^o's formula (see e.g. Remark 2 to Theorem IV.3.3 of \cite{RevuzYor99}) to the process
$e^{-(\rho + \delta) s}f_\infty\bar\Gamma(s)^{1-\gamma}$. Using then the HJB equation one would get the same result. We prefer not to use this path as the procedure would be less clear and intuitive for the reader.}
Hence we apply Proposition 1.164 in \cite{FABBRI_GOZZI_SWIECH_BOOK} obtaining
\begin{multline}\label{EXP_VER_INF_RETnew}
\begin{split}
&e^{-(\rho + \delta) (T\wedge\tau_N) }
\bar v\big(W_{\pi}(T\wedge\tau_N), X(T\wedge\tau_N)\big)
- \bar v\big(w,x\big)=
\\
=& \int_0^{T\wedge\tau_N} e^{-(\rho + \delta) s }\left\{   -(\rho + \delta)   \bar v\big(W_{\pi}(s), X(s)\big)  \right.
\\
&+\bar v_w \big(W_{\pi}(s), X(s)\big)
\big[ (r+ \delta) W_{\pi}(s)+ \theta^\top (s) (\mu-r\mathbf 1)
+  X_0(s) - c(s) - \delta B(s) \big]
\\
&+ \langle AX(s),\bar v_x \big(W_{\pi}(s), X(s)\big)  \rangle_{M_2}
+\frac{1}{2}\bar v_{ww} \big(W_{\pi}(s), X(s)\big)  \theta^\top \sigma \sigma^\top \theta
\\
&+\frac{1}{2} \bar v_{x_0 x_0} \big(W_{\pi}(s), X(s)\big)
\sigma_y^\top \sigma_y X_0^2 (s)
+ \bar v_{w x_0} \big(W_{\pi}(s), X(s)\big)
\theta^\top(s) \sigma\sigma_y X_0(s)\left.\right\} ds
\\
&+ \int_0^{T\wedge\tau_N} e^{-(\rho + \delta) s }\left\{
\bar v_{w} \big(W_{\pi}(s), X(s)\big)  \theta^\top(s)
\sigma + \bar v_{x_0} \big(W_{\pi}(s), X(s)\big) X_0 (s)  \sigma_y^\top
\right\}dZ_s.
\end{split}\end{multline}
By Proposition \ref{PROP_COMPARISON_FINITENESS_VAL_FUN_INF_RET},
the function $\bar v $ solves the HJB equation \eqref{eq:HJB1}
%(\ref{FORMAL_HJB_INF_RET}):
thus, for any $s\in [0,T\wedge \tau_N]$ we get
$$   (\rho+\delta )  \bar v(W_{\pi}(s), X(s))
= \mathbb H\big(W_{\pi}(s),X(s),
 D\bar v(W_{\pi}(s), X(s)),D^2 \bar v(W_{\pi}(s), X(s))\big).$$
By this last observation, and recalling (\ref{DEF_H_1_INF_RET}) and (\ref{DEF_H_MAX_INF_RET}),
the equality in (\ref{EXP_VER_INF_RETnew}) can be rewritten as
 \begin{eqnarray}\label{EXP_VER_INF_RET_II}
 \begin{split}
&e^{-(\rho + \delta) (T\wedge\tau_N) }
\bar v\big(W_{\pi}(T\wedge\tau_N), X(T\wedge\tau_N)\big)
- \bar v\big(w,x\big)=
\\[2mm]
=& - \int_0^{T\wedge\tau_N} e^{-(\rho + \delta) s }\left\{
\mathbb H_{max}\big(
W_{\pi}(s),X(s),D\bar v(W_{\pi}(s),X(s)),D^2 \bar v(W_{\pi}(s),X(s))
\big)\right.
\\[2mm]
&\qquad\qquad\left.- \mathbb H_{cv}\big(
W_{\pi}(s),X(s),D\bar v(W_{\pi}(s),X(s)),D^2\bar v(W_{\pi}(s),X(s));\pi(s)
\big)
\right\} ds
\\[2mm]
& -  \int_0^{T\wedge\tau_N}  e^{-(\rho + \delta) s }   \left( \frac{ c(s)^{1-\gamma}}{1-\gamma}
+ \delta \frac{\big(k B(s)\big)^{1-\gamma}}{1-\gamma}\right)    ds
\\[2mm]
&+ \int_0^{T\wedge\tau_N} e^{-(\rho + \delta) s }\left\{\bar v_{w} \big(W_{\pi}(s), X(s)\big) \theta^\top(s) \sigma
 + \bar v_{x_0} \big(W_{\pi}(s), X(s)\big)X_0 (s)  \sigma_y^\top    \right\}                dZ_s.
\end{split}\end{eqnarray}
By definition of $\mathbb H_{max}$ and $\mathbb H_{cv}$, we have
\begin{eqnarray}\label{POSITIVINESSfundid}
\mathbb H_{max}- \mathbb H_{cv} \geq 0.
\end{eqnarray}
Moreover the integral with respect to the Brownian motion in (\ref{EXP_VER_INF_RET_II}) is a martingale.
Taking the expectation we then get
   \begin{multline}\label{EXP_VER_INF_RET_IIImore}
 \begin{split}
&\mathbb E\left[e^{-(\rho + \delta) (T \wedge \tau_N) }
 \bar v\big(W_{\pi}(T \wedge \tau_N), X(T \wedge \tau_N)\big) \right]
- \bar v\big(w,x\big)=
\\[2mm]
=&-\mathbb E  \int_0^{T\wedge\tau_N} e^{-(\rho + \delta) s }\left\{
\mathbb H_{max}\big(
W_{\pi}(s),X(s),D\bar v(W_{\pi}(s),X(s)),D^2 \bar v(W_{\pi}(s),X(s))
\big)\right.
\\[2mm]
&\qquad\qquad\left.- \mathbb H_{cv}\big(
W_{\pi}(s),X(s),D\bar v(W_{\pi}(s),X(s)),D^2\bar v(W_{\pi}(s),X(s));\pi(s)
\big)
\right\} ds
\\[2mm]
& - \mathbb E \int_0^{T\wedge\tau_N}  e^{-(\rho + \delta) s }   \left( \frac{ c(s)^{1-\gamma}}{1-\gamma}
+ \delta \frac{\big(k B(s)\big)^{1-\gamma}}{1-\gamma}\right)    ds {.}
\end{split}\end{multline}
%Now we treat separately the two cases, $\gamma \in (0,1)$ and $\gamma>1$.
%Let first $\gamma \in (0,1)$ and
Now we let $N \rightarrow +\infty$.
%and then $T\to +\infty$.
The first expectation converges thanks to the dominated convergence theorem. Moreover the integrands on the right hand side are both nonnegative.
Thus we can apply to them the monotone convergence theorem. Their limits must be finite since also the left hand side is finite. Hence
%The expectation on the left hand side converges since all the other terms converge.
%We then get
   \begin{multline}\label{EXP_VER_INF_RET_IIInew}
 \begin{split}
&\bar v\big(w,x\big)=
\mathbb E \int_0^{T\wedge \tau}  e^{-(\rho + \delta) s }   \left( \frac{ c(s)^{1-\gamma}}{1-\gamma}
+ \delta \frac{\big(k B(s)\big)^{1-\gamma}}{1-\gamma}\right)ds
\\[2mm]
+&\mathbb E  \int_0^{T\wedge \tau} e^{-(\rho + \delta) s }\left\{
\mathbb H_{max}\big(
W_{\pi}(s),X(s),D\bar v(W_{\pi}(s),X(s)),D^2 \bar v(W_{\pi}(s),X(s))
\big)\right.
\\[2mm]
&\qquad\qquad\left.- \mathbb H_{cv}\big(
W_{\pi}(s),X(s),D\bar v(W_{\pi}(s),X(s)),D^2\bar v(W_{\pi}(s),X(s));\pi(s)
\big)
\right\} ds
\\[2mm]
&+
\E\left[e^{-(\rho + \delta) (T \wedge \tau) }
 \bar v\big(W_{\pi}(T \wedge \tau), X(T \wedge \tau)\big) \right].
\end{split}\end{multline}
Let now $T\to +\infty$.
The last term above converges to $0$ thanks to
\eqref{EQ_DECREASING_PROD_INF_RETlim}.
The two integrals converge again thanks to monotone convergence and their limits are finite since the left hand side is finite. Hence we get
   \begin{multline}\label{EXP_VER_INF_RET_IIInewmore}
 \begin{split}
&\bar v\big(w,x\big)=
\mathbb E \int_0^{\tau}  e^{-(\rho + \delta) s }   \left( \frac{ c(s)^{1-\gamma}}{1-\gamma}
+ \delta \frac{\big(k B(s)\big)^{1-\gamma}}{1-\gamma}\right)ds
\\[2mm]
+&\mathbb E  \int_0^{\tau} e^{-(\rho + \delta) s }\left\{
\mathbb H_{max}\big(
W_{\pi}(s),X(s),D\bar v(W_{\pi}(s),X(s)),D^2 \bar v(W_{\pi}(s),X(s))
\big)\right.
\\[2mm]
&\qquad\qquad\left.- \mathbb H_{cv}\big(
W_{\pi}(s),X(s),D\bar v(W_{\pi}(s),X(s)),D^2\bar v(W_{\pi}(s),X(s));\pi(s)
\big)
\right\} ds
\end{split}\end{multline}
By the definition of $\tau$ and Lemma \ref{lm:Gammabar}-(ii)
it follows that
$$
J(w,x;\pi)=\E \int_0^{\tau}  e^{-(\rho + \delta) s }   \left( \frac{ c(s)^{1-\gamma}}{1-\gamma}
+ \delta \frac{\big(k B(s)\big)^{1-\gamma}}{1-\gamma}\right)ds <+\infty
$$
and the claim follows.
\end{proof}

\bigskip

By the statement of Proposition \ref{pr:fundid} we get the following.
\begin{corollary}\label{cr:FINITENESS_VALUE_FUNCTION}
Let $\gamma \in (0,1)$. Then the value function $V$ is finite on $\calh_+$
and $V(w,x)\le \bar v(w,x)$ for every $(w,x)\in \calh_+$.
\end{corollary}

\begin{proof}
It is enough to observe that the integrand in \eqref{EXP_VER_INF_RET_III}
is positive, hence, we have, for every $(w,x)\in \calh_+$ and $\pi=(c,B,\theta)\in \Pi(w,x)$,
$$
\bar v(w,x) \ge J(w,x;\pi).
$$
Since ${V}(w,x) = \sup_{\pi\in \Pi(w,x)}J(w,x;\pi)$, the claim immediately follows.
\end{proof}

%We have seen in the proof of the second statement of Proposition \ref{PROP_COMPARISON_FINITENESS_VAL_FUN_INF_RET}
%that the function $\bar v$ defined in (\ref{EQ_GUESS_BAR V_INF_RET}) is a classical solution of the HJB.

{\color{black}
\begin{remark}\label{rm:newuniqueness}
Observe that, from the proof the above Proposition \ref{pr:fundid}, we easily obtain that the fundamental identity \eqref{EXP_VER_INF_RET_III} holds when, in place of $\bar v$, we put any classical solution $v$ of the HJB equation \eqref{eq:HJB1} which satisfies ($\tau$ being the first exit time from $\calh_{++}$),
\begin{equation}\label{eq:trasvnew}
\lim_{T\to + \infty}\E\left[e^{-(\rho + \delta) (T \wedge \tau) }
v\big(W_{\pi}(T \wedge \tau), X(T \wedge \tau)\big) \right]=0.
\end{equation}
A sufficient condition for such limit to hold is the following:
$v=0$ on $\partial \calh_{++}$ and
$|v(w,x)| \le C(1+|w|^{1-\gamma}+|x|^{1-\gamma})$ on $\calh_{++}$. The latter indeed allows us to prove that \eqref{EQ_DECREASING_PROD_INF_RETlim} holds for $v$ with
essentially the same proof provided in Lemma \ref{LEMMA_LIMIT_AT_INFTY_GUESS_VALUE_FUNCTION_INFINITE_RETIREMENT}
for $\bar v$.
\\
The above implies that Corollary \ref{cr:FINITENESS_VALUE_FUNCTION}
holds for $v$, too.
This proves a one-side comparison result: every classical solution of the HJB equation
\eqref{eq:HJB1}, satisfying \eqref{eq:trasvnew}, is larger than the value function $V$.
\\
One may then ask if classical solutions (satisfying the boundary condition
$v=0$ on $\partial \calh_{++}$) are unique. In state constraints problem like this one, the answer is, in general, negative (see e.g. \cite[Remark 6.2]{BarucciGozziSwiech00} for a simple example in one dimension).
The best one can expect in such cases is one-side comparison results of the type just mentioned (see e.g. \cite{KocanSoravia98,Soravia99}).
\\
The reason of this fact lies exactly in the presence of state constraints.
Indeed, if no constraints are present, we can take any classical solution $v$
of the HJB equation and use identity \eqref{EXP_VER_INF_RET_III} for such $v$.
It is clear that, if we find an admissible control such that the
integrand in \eqref{EXP_VER_INF_RET_III} is zero, then this control is optimal and we have $v=V$. Without state constraints this would be true possibly under further regularity assumptions (in this case, for example, by requiring $v_w>0$ and $v_{ww}<0$) to give sense to the feedback map and obtain a solution of the closed loop equation. In the presence of state constraints, however, it may happen that for any regular solution $v$ we find a control that makes the integrand in \eqref{EXP_VER_INF_RET_III} vanish but is not admissible.
\\
Clearly this leaves the possibility that, in this particular case, uniqueness hold in some form. However, this is not needed to solve our problem as we will see in the next subsections.
\end{remark}
}

\subsection{Verification Theorem and optimal feedbacks when $\gamma \in(0,1)$}
\label{SSE:VT01}

Now we proceed to show that
 %\ref{THM_VERIFICATION_THEOREM_INF_RET} below that
$\bar v=V$ and to find the optimal strategies in feedback form{.}

First, we provide the following definitions.

\begin{definition}\label{DEF_ADMISSIBLE_FEEDBACK STRATEGY_INF_RET}
Fix $(w,x) \in \calh_+$. A strategy $\bar \pi:=\left(\bar c,\bar B, \bar \theta \right)$
is called an \textit{optimal strategy} if
$\left(\bar c,\bar B, \bar \theta \right) \in \Pi\left(w,x\right)$,
and if it achieves the supremum in (\ref{DEF_VALUE_FUNCTION_INF_RET}),
i.e.
\begin{align}
V\left(w,x\right)=  \mathbb E \left(\int_{0}^{+\infty} e^{-(\rho+ \delta) s }
\left( \frac{\bar c(s)^{1-\gamma}}{1-\gamma}
+ \delta \frac{\big(k \bar B(s)\big)^{1-\gamma}}{1-\gamma}\right) ds
\right) .
\end{align}
\end{definition}

\begin{definition}
We say that a function
$\left( \textbf{C}, \textbf{B}, \Theta\right): \calh_+ \longrightarrow \mathbb R_{+}\times \mathbb R_{+}\times \mathbb R^{n}$
is an \textit{optimal feedback map} if, for every $(w,x)\in \calh_+$
the closed loop equation
\begin{eqnarray}%\label{}
\begin{split}
\left\{\begin{array}{ll}
dW(t) =&  \left[ (r+\delta) W(t)+\Theta^\top\left( W(t), X(t)\right) (\mu-r)  +  X_0(t) - \textbf{C}\left( W(t), X(t)\right) - \delta\textbf{B}\left( W(t), X(t)\right)\right] dt  \\
  &+  \Theta^\top \left( W(t), X(t)\right) \sigma dZ(t),\\
  W(0)=& w,
 \end{array}\right. \end{split}
\end{eqnarray}
coupled with the second of (\ref{DYN_W_X_INFINITE_RETIREMENT_II}), i.e.
\begin{equation}\label{eq:delaySEnew}
dX(t)=AX(t)+(CX(t))^\top dZ_t,\qquad X(0)=x,
\end{equation}
has a unique solution $(W^*,X)$, and
 the associated control strategy
$\left(\bar c, \bar B, \bar \theta \right)$
\begin{align}
\begin{split}
\left\{\begin{array}{ll}
\bar c(t)&:=  \textbf{C}\left(W^*(t),X(t)\right) {,} \\
\bar B(t)&:=\textbf{B}\left( W^*(t),X(t)\right){,}\\
\bar \theta(t)&:=\Theta    \left( W^*(t),X(t)\right)
\end{array}\right. \end{split}
\end{align}
is an optimal strategy.
\end{definition}

%Given $(W,X) \in \mathbb R \times M_2 $ we will call
%$(c, B, \theta) \in \mathcal C (\mathbb R \times M_2 ; \mathbb R_+^2 \times \mathbb R^n)$
%%
%\footnote{$\mathcal C (\mathbb R \times M_2; \mathbb R_+^2 \times \mathbb R^n)$
%denotes the set of continuous functions from $\mathbb R \times M_2$
%in $\mathbb R_+^2 \times \mathbb R^n$. }
%%
%a \emph{closed-loop strategy} related to the initial point $(W,X)  $ if
%\begin{equation*}
%	\begin{split}
% dW(t) = & \left[W(t) r+ \theta^\top(W(t),X(t)) (\mu-r\mathbf 1)  +  X_0(t) - c(W(t),X(t)) \right.\\
%  &\left.+ \delta \big(W(t)-B(W(t),X(t))\big) \right] dt  \\
%  &+  \theta^\top(W(t),X(t)) \sigma dZ(t)\quad \forall t>0,\\
%   W(0)=&W,\quad \quad X_0(0)=X,
%\end{split}
%\end{equation*}
%has a unique solution $W(t; c,B,\theta)$, and
%$\Big(c\big(W(t; c,B,\theta),X(t)\big)$, $B\big(W(t; c,B,\theta),X(t)\big)$, $\theta\big(W(t; c,B,\theta),X(t)\big)\Big)$ is in $\Pi$.
%%$\big(c(W_{(c, B, \theta)}(t),X(t)), B(W_{(c, B, \theta)}(t),X(t)), \theta(W_{(c, B, \theta)}(t),X(t))\big)$ is in $\Pi^0_{\infty}$.
%
%
Recalling that
\begin{equation}\label{Gamma-infty}
\Gamma_{\infty} (w,x): =w+ g_{\infty}x_0+ \langle h_{\infty}, x_1 \rangle,
\end{equation}
we define
%will show in Theorem \ref{THM_VERIFICATION_THEOREM_INF_RET} below
%that
the map
\begin{align}\label{EQ_DEF_FEEDBACK_MAP}
\begin{split}
\left\{\begin{array}{ll}
\textbf{C}_{f}( w,x)&:=   f_{\infty}^{-1}\Gamma_{\infty}( w,x){,}    \\
\textbf{B}_{f}( w,x)&:=k^{ -b} f_{\infty}^{-1}\Gamma_{\infty}( w,x){,}  \\
\Theta_f( w,x)&:=(\sigma\sigma^\top)^{-1} (\mu-r\mathbf 1) \frac{\Gamma_{\infty} ( w,x)}{\gamma  }-   (\sigma\sigma^\top)^{-1} \sigma   \sigma_y  g_{\infty} x_0
\\[1.5mm]
&=
\frac{1}{\gamma  }\Gamma_{\infty} ( w,x)(\sigma^\top)^{-1} \kappa
- g_{\infty} x_0(\sigma^\top)^{-1}  \sigma_y {.}
\end{array}\right.\end{split}
\end{align}
Observe that this map is obtained taking the maximum points of the Hamiltonian given in \eqref{MAX_POINTS_HAMILTONIAN__INF_RET} and substituting, in place of $p_1, P_{11}, P_{12}$
the derivatives $\bar v_w, \bar v_{ww}, \bar v_{w,x_0}$, respectively.

We aim to prove that this is an optimal feedback map.
For given $(w,x)\in \calh_+$, denote with $W_f^*(t)$ the unique solution of the associated closed loop equation
(coupled with \eqref{eq:delaySEnew}),
\begin{eqnarray}\label{CLOSED_LOOP_EQUATION_W}
\begin{split}
\left\{\begin{array}{ll}
dW(t) =&  \left[ (r+\delta) W(t)+ \Theta_f^\top\left(W(t),X(t)\right) (\mu-r\mathbf 1)  +  X_0(t) - \textbf{C}_f\left(W(t),X(t)\right) - \delta \textbf{B}_f\left(W(t),X(t)\right) \right] dt  \\
  &+  \Theta_f^\top (t) \sigma dZ(t),\\
  W(0)=& w,
 \end{array}\right. \end{split}
\end{eqnarray}
and set
\begin{equation}\label{DEF_GAMMA_INFTY_STAR}
\Gamma_{\infty}^*(t)= \Gamma_{\infty} \big(W_f^*(t), X(t)\big)  =W_f^*(t) + g_{\infty}X_0(t)+ \langle h_{\infty}, X_1(t)\rangle.
\end{equation}
%By Theorem \ref{THM_VERIFICATION_THEOREM_INF_RET} we will see that
%$\Gamma_{\infty}^*(t)$ is the total wealth when running the optimal strategy.
The control strategy associated with (\ref{EQ_DEF_FEEDBACK_MAP}) is
\begin{eqnarray}\label{EQ_FEEDBACK_STRATEGIES_INF_RET}
\left\{\begin{split}
\bar c_f(t)&:= \textbf{C}_f\left( W^*_f(t),X(t)\right) =  f_{\infty}^{-1}
\Gamma_{\infty}^*(t) ,  \\
\bar B_f(t)&:=\textbf{B}_f \left( W^*_f(t),X(t)\right) =k^{ -b} f_{\infty}^{-1}
\Gamma_{\infty}^*(t)  ,  \\
 \bar \theta_f(t)&:= \Theta_f\left( W^*_f(t),X(t)\right) =
\frac{\Gamma^*_{\infty} (t)}{\gamma  }(\sigma^\top)^{-1} \kappa
- g_{\infty} X_0(t)(\sigma^\top)^{-1}  \sigma_y{.}
%(\sigma\sigma^\top)^{-1} (\mu-r\mathbf 1) \frac{\Gamma_{\infty}^*(t) }{\gamma  }-    (\sigma^{-1})^\top   \sigma_y  g_{\infty} X_0(t).
\end{split}
\right.
\end{eqnarray}
%\begin{definition}\label{DEF_OPTIMAL_FEEDBACK STRATEGY_INF_RET}
%Given $(W,X) \in \mathbb R \times M_2$ we will call
%$(c, B, \theta) \in \mathcal C (\mathbb R \times \mathcal H; \mathbb R_+^2 \times \mathbb R^n)$
%CECILIA-DEVONO ESSERE LOSED LOOP STRATEGY
%an \emph{optimal closed-loop strategy} related to the initial point $(W,X)  $ if
%\begin{equation}
%\begin{split}
%V(W,X) =
%\mathbb E \left\{\int_{0}^{+\infty} e^{-(\rho+ \delta) s }
%\left( \frac{\Big(c\big(W^{W,X}(s; c,B,\theta),X^X(s)\big)\Big)^{1-\gamma}}{1-\gamma}\right. \right.\\
%\left.\left.
%+ \delta \frac{\Big(k B\big(W^{W,X}(s; c,B,\theta),X^X(s)\big)\Big)^{1-\gamma}}{1-\gamma}\right) ds\right\}.
%\end{split}
%\end{equation}
%%
%\noindent Using the notation
%\begin{equation}\label{Gamma-infty}
%\Gamma_{\infty}=\Gamma_{\infty} (W,X): =W + g_{\infty}X_0+ \langle h_{\infty}, X_1 \rangle,
%\end{equation}
%define the closed-loop strategy
%\begin{align}\label{EQ_FEEDBACK_STRATEGIES_INF_RET}
%\begin{split}
%\hat c(W,X)&:=   M^{ -b} f_{\infty}^{-1}\Gamma_{\infty}(W,X) ,  \\
%\hat B(W,X)&:=k^{ -b} f_{\infty}^{-1}\Gamma_{\infty}(W,X) ,  \\
% \hat \theta(W,X)&:=   (\sigma\sigma^\top)^{-1} (\mu-r\mathbf 1) \frac{\Gamma_{\infty}(W,X) }{\gamma  }-    \sigma^{-1}   \sigma_y  g_{\infty} X_0.
%\end{split}
%\end{align}
%\end{definition}
% %
Note that $\Gamma_{\infty}^*(t)$ is the total wealth associated to the strategy
$(\bar c_f, \bar B_f, \bar \theta_f)$.
We need the following lemma which is needed to show the admissibility of such strategy.

\begin{lemma}\label{PROP_DYNAMIC_H_INF_RET}
Let $(w,x)\in \calh_+$. The process $\Gamma_{\infty}^*$ defined in (\ref{DEF_GAMMA_INFTY_STAR}) is a stochastic exponential,
and it has dynamic
\begin{align}\label{DYN_H^*_INFTY}
\begin{split} d  \Gamma_{\infty}^* (t) =& \Gamma_{\infty}^* (t) \Big(  r + \delta +\frac{1}{\gamma} \kappa^\top \kappa- f_{\infty}^{-1}\big(1 %M^{-b}
+\delta k^{-b}\big) \Big)dt%\\
%&
+ \frac{ \Gamma_{\infty}^* (t)}{\gamma } \kappa^\top dZ(t).
\end{split}
\end{align}
\end{lemma}
\begin{proof}
Substituting \eqref{EQ_FEEDBACK_STRATEGIES_INF_RET} into the equation
\eqref{CLOSED_LOOP_EQUATION_W} for $W_f^*(t)$, we get
\begin{align}\label{CLOSED_LOOP_W_INF_RET}
\begin{split}
d W_f^*(t) =&\Big\{ W_f^*(t)(r+\delta) +
 \Gamma_{\infty}^*(t)\big[ \frac{\kappa^\top \kappa}{\gamma }   - f_{\infty}^{-1}\big(1 %M^{-b}
 + \delta k^{-b}\big)  \big]
 + X_0(t)
-   g_{\infty} X_0(t) \sigma_y^\top\kappa
\Big\}dt\\
&+ \left\{ \frac{\Gamma_{\infty}^*(t)}{ \gamma} \kappa^\top - g_{\infty} X_0(t)\sigma_y^\top  \right\}dZ(t).
\end{split}
\end{align}
\noindent
%Recalling %(\ref{DIFFERENTIAL_G_X_0_H_X_1_INF_RET}) {\red [CHECK]} in
%the proof of
%Proposition~\ref{PROP_CONSTRAINT_INTERPRETATION_INF_RET},
%and noting that %from \cite{BGP} we have
Then, recalling \eqref{eq:hinftyX1} we have
\begin{align}\label{DIFFERENTIAL_G_X_0_H_X_1_INF_RET}
\begin{split}
d\langle (g_{\infty},h_{\infty}),& (X_0(t),X_1(t)) \rangle_{M_2}=
\\
&
= \big[\big(\mu_y + \beta\big) g_{\infty} X_0(t) -X_0(t)+ (r+\delta)\langle h_{\infty},   X_1(t) \rangle
\big]  dt+   g_{\infty}  X_0(t)  \sigma_y^\top dZ_t,
\end{split}
\end{align}
%where the last equality follows by (\ref{EQ_g_h_infty}).
We thus obtain, similarly to \eqref{eq:bargammaito},
\begin{align}\label{CLOSED_LOOP_W_FINITE_HOR}
\begin{split}
d \Gamma_{\infty}^*(t) & =  d W_f^{*}(t) + d \langle (g_{\infty}, h_{\infty}), \big(X_0(t),X_1(t)\big)\rangle_{M_2} \\
=&\bigg\{W_f^{*}(t)(r+\delta) +
 \Gamma_{\infty}^*(t)\left[  \frac{\kappa^\top \kappa}{\gamma}   - f_{\infty}^{-1}\big(1 %M^{-b}
 + \delta k^{-b}\big)  \right]
 + X_0(t)
- g_{\infty}X_0(t)\sigma_y^\top\kappa   \\
&+ ( \mu_y + \beta) g_{\infty} X_0(t)-X_0(t) + (r+\delta)  \langle   h_{\infty},X_1(t)\rangle
\bigg\}dt\\
&+ \left\{\frac{\Gamma_{\infty}^*(t) }{ \gamma }\kappa^\top - g_{\infty}X_0(t) \sigma_y^\top
+      g_{\infty} X_0 (t)\sigma_y ^\top \right\}dZ(t)\\
=& \Gamma_{\infty}^*(t)\left[ (r+\delta)+ \frac{\kappa^\top \kappa}{\gamma}   - f_{\infty}^{-1}\big(1 %M^{-b}
+ \delta k^{-b}\big)  \right] dt
+ \frac{\Gamma_{\infty}^*(t)}{ \gamma} \kappa^\top dZ(t) ,
\end{split}
\end{align}
where the last equality follows by noting
that $\mu_y + \beta - \sigma_y^\top \kappa = r+\delta$. The claim is proved.
\end{proof}
\begin{theorem}[Verification Theorem and Optimal feedback Map, $\gamma\in (0,1)$] \label{th:VERIFICATION_THEOREM_INF_RET}
Let $\gamma \in (0,1)$. We have $V=\bar v$ in $\calh_+$.
Moreover the function $\left(\textbf{C}_f, \textbf{B}_f, \Theta_f\right)$ defined in (\ref{EQ_DEF_FEEDBACK_MAP}) is an optimal feedback map.
Finally, for every $(w,x)\in \calh_{+}$ the strategy
$\bar\pi_f:=(\bar c_f,\bar B_f,\bar\theta_f)$ is the unique optimal strategy.
%and the strategy
%\begin{eqnarray}
%\begin{split}
%\left\{\begin{array}{ll}
%\end{array} \right.
%\hat c(t):= C_f\left(W^*(t)\right),\\
%\hat B:=,\\
% \hat \theta:=,
%\end{split}
%\end{eqnarray}
\end{theorem}

\begin{proof}
%[Proof of Theorem \ref{THM_VERIFICATION_THEOREM_INF_RET}]
First take $(w,x) \in \partial\calh_+=\{\Gamma_\infty=0\}$.
In this case we have, by equation \eqref{DYN_H^*_INFTY},
for every $t\ge 0$, $\Gamma^*_{\infty}(t)=0$, $\P$-a.s..
This implies, by \eqref{EQ_FEEDBACK_STRATEGIES_INF_RET}, that it must be
$$
\bar c_f \equiv 0,\qquad
\bar B_f \equiv 0,\qquad
\bar \theta_f \equiv - g_{\infty} X_0(t)(\sigma^\top)^{-1}  \sigma_y {.}
$$
Hence, by Lemma \ref{lm:Gammabar}-(ii), we get that this is the
only admissible strategy (see \eqref{eq:zerostrategy}),
hence it is optimal.

Now take $(w,x) \in \calh_{++}=\{\Gamma_\infty>0\}$.
First we observe that $(\bar c_f, \bar B_f,\bar \theta_f)$ is an admissible
strategy.
Indeed
 %i.e. $(\bar c_f, \bar B_f,\bar \theta_f) \in \Pi(w,x)$,
%and then we will show that $(\bar c_f, \bar B_f,\bar \theta_f)$ is an optimal strategy
%and that $\bar v$ is indeed the value function.
%\begin{itemize}
%\item
%The closed-loop equation (\ref{CLOSED_LOOP_W_INF_RET}) in the proof of
%Lemma \ref{PROP_DYNAMIC_H_INF_RET} has a unique solution.
by Lemma \ref{PROP_DYNAMIC_H_INF_RET} $\Gamma_{\infty}^*(\cdot)$ is a stochastic exponential, it is therefore $\P$-a.s. strictly positive for any strictly positive initial condition $\Gamma_{\infty}^*(0)=\Gamma_\infty(w,x)$.
This implies, in particular, that the constraint in (\ref{NO_BORROWING_WITHOUT_REPAYMENT_EXPLICIT_INF_RET})
is always satisfied and, consequently, that the couple $(\bar c_f,\bar B_f)$ is non negative, which is enough to prove admissibility.

Concerning optimality we observe that, as recalled above,
the feedback map is obtained taking the maximum points of the Hamiltonian given in \eqref{EQ_DEF_FEEDBACK_MAP} and substituting, in place of $p_1, P_{11}, P_{12}$
the derivatives $\bar v_w, \bar v_{ww}, \bar v_{w,x_0}$, respectively.

This implies that, substituting the strategy
$\bar\pi_f:=(\bar c_f,\bar B_f,\theta_f)$
in the fundamental identity \eqref{EXP_VER_INF_RET_III} we obtain
$$
\bar v(w,x)=J\left(w,x;\bar\pi_f\right){.}
$$
Hence, using Corollary \ref{cr:FINITENESS_VALUE_FUNCTION} and the definition of the value function{,} we get
$$
V(w,x)\le \bar v(w,x)=J\left(w,x;\bar\pi_f\right)\le V(w,x){,}
$$
which immediately gives $ V(w,x)=J\left(w,x;\bar\pi_f\right)$, hence the required optimality.

We now prove uniqueness. When $(w,x)\in \partial\calh_+$ the claim follows from Lemma \ref{lm:Gammabar}-(ii). When $(w,x)\in \calh_{++}$,
the claim follows from the fundamental identity \eqref{EXP_VER_INF_RET_III}.
Indeed, since $\bar v=V$ if a given strategy $\pi$ is optimal at
$(w,x)\in \calh_{++}$ it must satisfy
$\bar v(w,x)=J(w,x;\pi)$, which implies, substituting in \eqref{EXP_VER_INF_RET_III},
that the integral in \eqref{EXP_VER_INF_RET_III} is zero. This implies that, on $[0,\tau]$ we have $\pi=\bar\pi_f$, $dt \otimes \P$-a.e. This gives uniqueness, as, for $t> \tau$, we still must have $\pi=\bar\pi_f$, $dt \otimes \P$-a.e., due to Lemma
\ref{lm:Gammabar}-(ii).
\end{proof}

\subsection{The case $\gamma >1$}
\label{SSE:GAMMA>1}

We now treat the case when $\gamma>1$.
We cannot follow the same path as done in the case $\gamma \in (0,1)$.
Indeed the proof of the crucial limiting condition \eqref{EQ_DECREASING_PROD_INF_RETlim} in Lemma \ref{LEMMA_LIMIT_AT_INFTY_GUESS_VALUE_FUNCTION_INFINITE_RETIREMENT} does not work as it is, since Fatou's Lemma cannot be applied here. Indeed in such proof, we use Fatou's Lemma with the liminf inequality: this requires a uniform bound from below
which we are not able to prove here.\footnote{In this respect it seems that the proof of Proposition 4.26 in \cite{DFGFS} (case ($\gamma<0$)
is not completely correct as it uses Fatou's lemma in the wrong direction.}

We have to follow a different path.
We start by looking at the value function $V$.
%is finite and bigger than $\bar v$.
We already know that $-\infty \le V(w,x)\le 0$ for every $(w,x)\in \calh_+$
and that, by Lemma \ref{lm:Gammabar}-(ii), $V(w,x)=-\infty$ for every
$(w,x)\in \partial\calh_+$.

To prove that $V(w,x)>-\infty$ on $\calh_{++}$ it is enough to find an admissible strategy $\pi$ such that $J(w,x;\pi) > - \infty$. This is given in the following proposition.

\begin{proposition}\label{pr:Vfinitegamma>1}
Let $\gamma>1$ and let $(w,x)\in \calh_{++}$.
The strategy $\bar\pi_f:=(\bar c_f,\bar B_f,\bar\theta_f)$ defined in \eqref{EQ_FEEDBACK_STRATEGIES_INF_RET} is admissible at $(w,x)$. Moreover
\begin{equation}\label{eq:J=barv}
V(w,x)\ge J(w,x;\bar\pi_f)=\bar v(w,x) > - \infty.
\end{equation}
\end{proposition}
\begin{proof}
The admissibility follows from the proof of Theorem \ref{th:VERIFICATION_THEOREM_INF_RET} since the constraint to be satisfied is the same as in the case $\gamma \in (0,1)$.
The validity of \eqref{eq:J=barv} is achieved by direct computation, using \eqref{EQ_FEEDBACK_STRATEGIES_INF_RET} and the fact that $\Gamma_\infty^*(\cdot)$
is a stochastic exponential as from Lemma \ref{PROP_DYNAMIC_H_INF_RET}.
\end{proof}

The above Proposition  \ref{pr:Vfinitegamma>1} implies that
$$
\bar v\big(w,x\big)= J(w,x;\bar\pi_f)\le V(w,x)\le 0 {.}
$$
We now want to prove that
$$
\bar v\big(w,x\big)\ge V(w,x).
$$
To do this we look closely at the value function.
%we get that $\bar v=V$ and that $\bar\pi_f$ is optimal.
%To prove this we using a maximum principle type argument,
%which is simplified version of the comparison argument for viscosity solutions.
%that $V$ is a viscosity supersolution
%of the HJB equation \eqref{eq:HJB1} and that a partial comparison holds.
%
%Before we need some useful results.

\begin{proposition}[Homogeneity of $V$]
\label{pr:Vhom}
%Let $\gamma>1$.
The value function in $\calh_{++}$ satisfies the following
\begin{equation}\label{eq:Vpower}
V(w,x)= \eta \frac{\Gamma_\infty(w,x)^{1-\gamma}}{1-\gamma}, \qquad
\hbox{for some $\eta \ge 0$.}
\end{equation}
\end{proposition}
\begin{proof}
We divide the proof in two steps.

{\em Step I. If, for $(w_1,x_1),(w_2,x_2)\in\calh_{++}$, we have $\Gamma_\infty(w_1,x_1)=\Gamma_\infty(w_2,x_2)$, then it must be
$$
V(w_1,x_1)=V(w_2,x_2).
$$}

Indeed let $\pi_1:=(c_1,B_1,\theta_1) \in \Pi(w_1,x_1)$.
Consider the strategy $\pi_2:=(c_2,B_2,\theta_2)$
where $c_2=c_1$, $B_2=B_1$, and
$$
\theta_2(t)^\top \sigma+g_\infty X_0^{x_2}(t)\sigma_y^\top
=
\theta_1(t)^\top \sigma+g_\infty X_0^{x_1}(t)\sigma_y^\top {.}
$$
With this choice of $\pi_1,\pi_2$ it is clear that
$$
\Gamma_\infty(W^{w_1,x_1}(t;\pi_1),X^{x_1}(t))
=
\Gamma_\infty(W^{w_2,x_2}(t;\pi_2),X^{x_2}(t)), \quad \hbox{$t$ a.e. and $\P$-a.s.}
$$
since both processes satisfy equation \eqref{eq:GammaProcessSDE}
with the same initial condition.
Hence we also have $\pi_2\in \Pi(w_2,x_2)$.
Moreover we clearly have
$$
J(w_1,x_1;\pi_1)=
J(w_2,x_2;\pi_2){.}
$$
Since this construction can be done for every
$\pi_1\in \Pi(w_1,x_1)$ we immediately get
$V(w_1,x_1)\le V(w_2,x_2)$. The same argument also applies to show the opposite inequality, hence the claim follows.

{\em Step II. Homogeneity.}
Since, by the previous Proposition \ref{pr:Vfinitegamma>1},
$0\ge V>-\infty$ then we must have, for some $f:\R_+\to \R_-$,
\begin{equation}\label{eq:fnew}
V(w,x)=f(\Gamma_\infty(w,x)), \qquad \forall (w,x)\in \calh_{++}{.}
\end{equation}
Now, from the homogeneity of the
objective functional $J$ and from the linearity of the state equation
and of the constraints we get that $f$ must be $(1-\gamma)$-homogeneous.
Indeed, let $(w,x)\in \calh_{++}$ and $\pi\in \Pi(w,x)$. For $a>0$ we have, by linearity,
$$
W^{aw,ax}(t;a\pi)=aW^{w,x}(t;\pi), \quad and \quad
X^{ax}(t)=aX^x(t).
$$
Hence, by linearity of $\Gamma_\infty$, it must be
$a \pi \in \Pi(aw,ax)$, so $a\Pi(w,x)\subseteq \Pi(aw,ax)$. With the same argument
we can prove that also $a\Pi(w,x)\supseteq \Pi(aw,ax)$, hence the two sets are equal. We then have
$$
V(aw,ax)=\sup_{\pi\in \Pi(w,x)} J((aw,ax);a\pi)
=
a^{1-\gamma}\sup_{\pi\in \Pi(w,x)} J((w,x);\pi)
=
a^{1-\gamma}V(w,x){.}
$$
From the above we immediately get that the function $f$ in \eqref{eq:fnew}
is $(1-\gamma)$-homogeneous, hence a power. Since $V\le 0$
we must have $\eta \ge 0$ in \eqref{eq:Vpower}. The claim is proved.
\end{proof}

%\begin{proposition}[Continuity of $V$]
%\label{pr:Vcont}
%The value function is continuous in $\calh_{++}$.
%Moreover it is concave and increasing in the first variable.
%\end{proposition}
%\begin{proof}
%The monotonicity follows since $\Pi(w_1,x) \subseteq \Pi(w_2,x)$ for every
%$w_1<w_2$ in $\R$ and $x\in M_2$. The concavity follows
%by the linearity of the state equation and of the constraints, and by the
%concavity of $J$. The continuity follows from the concavity and from the fact that
%$V\ge \bar v$.
%\end{proof}

\begin{proposition}[Dynamic Programming Principle]
\label{PROP_DYN_PROG_INF_RET}
For any stopping time $\tau $ with respect to $\mathbb F$,
the value function $V$ satisfies the dynamic programming principle
\begin{equation}\label{DYNAMIC_PROGRAMMING_EQUATION_INF_RET}
\begin{split}
 V\left(w,x\right) =  \sup_{\left(c,B,\theta\right) \in   \Pi\left(w,x\right) }  \mathbb E \Bigg\{\int_{0}^{\tau} & e^{-(\rho+ \delta) s }
\left( \frac{c(s)^{1-\gamma}}{1-\gamma}
+ \delta \frac{\big(k B(s)\big)^{1-\gamma}}{1-\gamma}\right) ds\\
&+  e^{-(\rho+ \delta) \tau }V\left(W^{w,x}(\tau;c,B, \theta),X^{x}(\tau)\right)
\Bigg\}.
\end{split}
\end{equation}
\end{proposition}
\begin{proof}
See Theorem 3.70 in \cite{FABBRI_GOZZI_SWIECH_BOOK}.
The only differences are that, in such theorem, one has:
\begin{itemize}
  \item the horizon is finite;
  \item the current cost is assumed to be state dependent and uniformly bounded in the control.
\end{itemize}
The first difference is easily overcome by standard shift arguments as is done, e.g. in Section 2.4 of \cite{FABBRI_GOZZI_SWIECH_BOOK}.
The second difference can be resolved observing that, in the proof of
Theorem 3.70 in \cite{FABBRI_GOZZI_SWIECH_BOOK}, such boundedness is used to apply dominated convergence inside the integral (see equation (3.166), p.242 of \cite{FABBRI_GOZZI_SWIECH_BOOK}. In our case, due to the specific form of the functional, we can apply monotone convergence to get the same result.
\end{proof}

\begin{proposition}
\label{th:VsolHJB}
The value function $V$ is a classical solution of the HJB equation \eqref{eq:HJB1}
in $\calh_{++}$.
%$$
%V(w,x)\le \bar v(w,x), \qquad \forall (w,x)\in \calh_{++}.
%$$
\end{proposition}
\begin{proof}
The proofs uses exactly the same arguments
of the proof of {Theorem 2.41} in \cite{FABBRI_GOZZI_SWIECH_BOOK}.
\end{proof}

Now we substitute the explicit expression \eqref{eq:Vpower} into the HJB equation
\eqref{eq:HJB1}. It is immediate to get that we have equality only in two cases:
either when $V\equiv\bar v$ or when $V\equiv 0$.
We now show that we can exclude the second possibility.
First we give a simple lemma.

\begin{lemma}\label{lm:tauinf}
Let $\gamma>1$, let $(w,x)\in \calh_{++}$ and let
$\pi:=(c, B,\theta)\in \Pi(w,x)$ be such that $J(w,x;\pi) > - \infty$.
Let $\bar\Gamma_\infty(t):=\Gamma_\infty(W_\pi(t),X(t))$.
Let $\tau$ be the first exit time of the process $\bar\Gamma_\infty(\cdot)$
from the open set $\calh_{++}$.
Then it must be $\P(\tau= +\infty)=1$.
%Moreover we must have
%$$
%\E[\bar\Gamma(t)^{1-\gamma}]\le
%$$
\end{lemma}
\begin{proof}
Assume by contradiction that $\P(\tau= +\infty)<1$.
Then, for some $T_1>0$
we would have $\P(\tau \le T_1)>0$.
By Lemma \ref{lm:Gammabar}-(ii),
on the set $\{\tau\le T_1\}$ we have
$$
\int_{T_1}^{+\infty}
 e^{-(\rho + \delta) s }   \left( \frac{ c(s)^{1-\gamma}}{1-\gamma}
+ \delta \frac{\big(k B(s)\big)^{1-\gamma}}{1-\gamma}\right)ds =-\infty {.}
$$
Hence also $J(w,x;\pi) = - \infty$, a contradiction.
\end{proof}

\begin{proposition}
\label{th:Vnot0}
The value function $V$ is not always $0$ in $\calh_{++}$.
\end{proposition}
\begin{proof}
Assume by contradiction that
$V\equiv 0$ over $\calh_{++}$.
Then by dynamic programming principle \eqref{DYNAMIC_PROGRAMMING_EQUATION_INF_RET}
we get that, for every $(w,x) \in \calh_{++}$ and $t\ge 0$
\begin{equation}\label{DYNAMIC_PROGRAMMING_EQUATION_INF_RET0}
\begin{split}
0 =  \sup_{\left(c,B,\theta\right) \in   \Pi\left(w,x\right) }
\mathbb E \Bigg\{\int_{0}^{t} & e^{-(\rho+ \delta) s }
\left( \frac{c(s)^{1-\gamma}}{1-\gamma}
+ \delta \frac{\big(k B(s)\big)^{1-\gamma}}{1-\gamma}\right) ds
\Bigg\}.
\end{split}
\end{equation}
In particular we fix $T>0$ and we choose, for every $n \in \N$,
$(w_n,x_n)$ such that $\Gamma_\infty(w_n,x_n)< 1/n$,
$t_n=T$, and $\left(c_n,B_n,\theta_n\right)\in \Pi\left(w_n,x_n\right)$
such that
\begin{equation}\label{DYNAMIC_PROGRAMMING_EQUATION_INF_RET0new}
\begin{split}
-\frac 1n <
\mathbb E \Bigg\{\int_{0}^{T} & e^{-(\rho+ \delta) s }
\left( \frac{c_n(s)^{1-\gamma}}{1-\gamma}
+ \delta \frac{\big(k B_n(s)\big)^{1-\gamma}}{1-\gamma}\right) ds
\Bigg\}<0.
\end{split}
\end{equation}
Now we use equation \eqref{eq:Gamma0tau} for $t=T$,
$\left(c,B,\theta\right)=\left(c_n,B_n,\theta_n\right)$,
and take the expectation $\tilde\E_T$ under $\tilde\P_T$
(defined in \eqref{eq:Ptdef}).
In this way, using the previous lemma, the stochastic integral in \eqref{eq:Gamma0tau} disappears and we get
\begin{equation}\label{DYNAMIC_PROGRAMMING_EQUATION_INF_RET0newbis}
\begin{split}
 \tilde{\mathbb E}_T \Bigg\{\int_{0}^{T} & e^{-(r+ \delta) s }
\left( c_n(s)+ \delta B_n(s)\right) ds
\Bigg\}\le \Gamma_\infty(w_n,x_n)< 1/n.
\end{split}
\end{equation}
This implies that the sequences $c_n$ and $B_n$, since they are positive,
converge to $0$ in $L^1(\Omega \times [0,T],d\P\otimes dt)$.
Hence there exist subsequences $c_{n_k}$ and $B_{n_k}$
which converge a.e. to $0$ in $d\P\otimes dt$.
Since $\gamma>1$, this implies that the subsequences
$c_{n_k}^{1-\gamma}/(1-\gamma)$
and $B_{n_k}^{1-\gamma}/(1-\gamma)$
converge a.e. to $-\infty$ in $d\P\otimes dt$.
This contradicts \eqref{DYNAMIC_PROGRAMMING_EQUATION_INF_RET0}, so the claim follows.
\end{proof}

The next, final, theorem, is then a straightforward consequence of the results of the present subsection.
%Now we show that $V\le \bar v$.

\begin{theorem}[Verification Theorem and Optimal feedback Map, $\gamma>1$] \label{th:VERIFICATIONgamma>1}
Let $\gamma>1$. We have $V=\bar v$ in $\calh_+$.
Moreover the function $\left(\textbf{C}_f, \textbf{B}_f, \Theta_f\right)$ defined in (\ref{EQ_DEF_FEEDBACK_MAP}) is an optimal feedback map.
Finally, for every $(w,x)\in \calh_{+}$ the strategy
$\bar\pi_f:=(\bar c_f,\bar B_f,\bar\theta_f)$ is the unique optimal strategy.
%and the strategy
%\begin{eqnarray}
%\begin{split}
%\left\{\begin{array}{ll}
%\end{array} \right.
%\hat c(t):= C_f\left(W^*(t)\right),\\
%\hat B:=,\\
% \hat \theta:=,
%\end{split}
%\end{eqnarray}
\end{theorem}

{
\begin{remark}\label{rm:newuniquenessbis}
Observe that, similarly to the case $\gamma\in (0,1)$ (see Remark \ref{rm:newuniqueness}), also in this case from the proof the above Proposition \ref{pr:fundid}, we easily get that the fundamental identity \eqref{EXP_VER_INF_RET_III} holds when, in place of $\bar v$, we put any classical solution $v$ of the HJB equation \eqref{eq:HJB1} which satisfies ($\tau$ being the first exit time from $\calh_{++}$),
$$
\lim_{T\to + \infty}\E\left[e^{-(\rho + \delta) (T \wedge \tau) }
v\big(W_{\pi}(T \wedge \tau), X(T \wedge \tau)\big) \right]=0.
$$
This would provide a one-side comparison result also in this case.
However, due to the asymmetry explained at the beginning of this subsection, this condition may not be easy to check, so such comparison result may be less interesting here.
Moreover, as we said in Remark \ref{rm:newuniqueness},
we do not expect here uniqueness to hold, in general. This doe not affect the solution of our problem.
\end{remark}
}

\section{Discussion and extensions}\label{SE:DISCUSSION}
Going back to the portfolio choice problem presented in Section~\ref{Problem formulation}, we recall that the pair $(X_0(t),X_1(t))$ can be identified with $(y(t),y(t+s)_{s \in [-d,0]})$, where $y(t)$
denotes labor income at time $t$, and
%$x_1$ %$x_1(s)_{s \in [-d,0]}$
$y(t+s)_{s \in [-d,0]}$
is the path of labor income between time $t$ and $d$ units of time in the past.
We can then summarize the results of the previous sections in the following theorem,
{\color{black} which
also states the results for the original problem \eqref{DEF_OBJECTIVE FUNCTION_DEATH TIMEbar}
by considering the original reference filtration $\mathbb G$ and the admissible control space $\overline \Pi$ which is defined as the set of all controls in $\overline \Pi_0$ whose pre-death counterpart satisfies the state constraint
\eqref{NO_BORROWING_WITHOUT_REPAYMENT_CONDITIONLA_MEAN}.}
%{
\begin{theorem}\label{TEO_MAIN_INF_RET}
	The value function $V$ of Problem \ref{Problem1} is given by
	\begin{equation}
	V(w,x_0,x_1) =  \frac{ f_{\infty}^{\gamma} \left(
		w + g_{\infty} x_0 + \int_{-d}^0 h_{\infty}(s) x_1(s) \,ds
		\right)^{1-\gamma} }
	{1-\gamma},
	\end{equation}
	where $f_{\infty}$ is defined in (\ref{EXPRESSIONS_f_INF_RET})
and $\left(g_{\infty},h_{\infty}\right)$ in (\ref{DEF_g_infty_h_infty}).
Moreover for every $(w,x) \in \R\times M_2$
there exists a unique optimal strategy $\pi^*=(c^*,B^*,\theta^*)\in \Pi_0$ starting at $(w,x)$.
Such strategy can be represented as follows.
Denote total wealth by
\vspace{-0.3truecm}
\begin{equation}\label{Gamma-infty2}
  \Gamma_{\infty}^* (t): =W^*(t) + g_{\infty}y(t)+ \int_{-d}^0 h_{\infty}(s) y(t+s) \,ds,
\vspace{-0.3truecm}
  	\end{equation}
where $W^*(\cdot)$ is the solution of equation (\ref{DYN_W_X_INFINITE_RETIREMENT_II}) with initial datum $w$ and control $\pi^*$, whereas $y(\cdot)$ is the solution of the second equation in (\ref{DYNAMICS_WEALTH_LABOR_INCOME}) with datum $x=(x_0,x_1)\in M_2$.
Then, $\Gamma^*_\infty$ has  dynamics
	\begin{align}\label{DYN_GAMMA*_PB1}
	\begin{split}
	d  \Gamma_{\infty}^* (t) =& \Gamma_{\infty}^* (t) \Big(  r + \delta +\frac{\kappa^\top \kappa}{\gamma}
	- f_{\infty}^{-1}\big( 1 %M^{-b}
	+\delta k^{-b}\big) \Big)dt%\\
	%&
	+  \frac{\Gamma_{\infty}^* (t)}{\gamma } \kappa^\top dZ(t),
	\end{split}
	\end{align}
and the optimal strategy triplet $\pi^*=(c^*,B^*,\theta^*)$
for Problem \ref{Problem1} is given by
	\begin{align}\label{OPTIMAL_STRATEGIES_RET_INF}
	\begin{split}
	c^{*}(t)&:=  %M^{ -b}
	f_{\infty}^{-1}  \Gamma_{\infty}^*(t) {,}  \\
	B^{*}(t)&:=k^{ -b } f_{\infty}^{-1}  \Gamma_{\infty}^*(t)      {,} \\
	\theta^{*}(t)&:=
\frac{\Gamma_{\infty}^*(t)}{ \gamma}
(\sigma^\top)^{-1}\kappa
- g_{\infty}y(t)(\sigma^\top)^{-1} \sigma_y.
	\end{split}
	\end{align}
{\color{black} Finally, there exists a unique $\mathbb G$-adapted optimal strategy $\overline \pi^*=(\overline c^*,\overline B^*,\overline \theta^*)\in \overline\Pi$ coinciding with $\pi^*=(c^*,B^*,\theta^*)\in \Pi$ on $\{\tau_{\delta} \ge t\}$. The optimal controls in the original filtration $\mathbb G$ are given by
$$
\overline c^{*}(t)= 1_{\{\tau_{\delta}\ge t\}} c^{*}(t),
\qquad \overline B^{*}(t)= 1_{\{\tau_{\delta}\ge t\}} B^{*}(t),
\qquad
\overline \theta^{*}(t)= 1_{\{\tau_{\delta} \ge t\}} \theta^{*}(t),
$$
with associated total and financial wealth given by
$$
\overline\Gamma^*(t)= 1_{\{\tau_{\delta}\ge t\}} \Gamma^{*}(t),
\qquad
\overline W^*(t)= 1_{\{\tau_{\delta}\ge t\}} W^{*}(t),
$$
respectively.}
\end{theorem}

To understand the optimal solution, we note that quantity
\eqref{Gamma-infty2} represents the agent's total wealth, given by the sum of financial wealth and human capital. In line with \cite{bodie1992labor,DYBVIG_LIU_JET_2010}, the agent considers the capitalized value of future wages as if they were a traded asset.
The solution follows the logic of Merton \cite{MERTON}, in that the agent chooses constant fractions of total wealth to consume and leave as bequest. The same would apply to the risky assets allocation if the agent's labor income were uncorrelated with the financial market. As it is instead instantaneously perfectly correlated with the risky assets, a negative income hedging demand arises
(the term $- g_{\infty}y(t)(\sigma^\top)^{-1} \sigma_y$) reducing the allocation to the risky assets accordingly  (\cite{CV}).
 %by a term proportional to the regression  coefficient of labor income shocks on risky asset returns (\cite{CV}).
The riskier the human capital, the less aggressive the agent's asset allocation.

We look now at the relation
with the benchmark model with no delay, i.e. when $\phi=0$.
First, we  observe that the dynamics of $\Gamma^*_\infty$ is not influenced by the path dependent component of the model in the sense that, {\em ceteris paribus}, changing $\phi$ (and hence $y$) leaves the dynamics of the total wealth unchanged at each time.
Then, calling $\widetilde \Gamma_\infty$ the total wealth
when $\phi=0$, we have, for any initial point $(w,x)$
$$
\Gamma^*_\infty(0)-\widetilde\Gamma_\infty(0)
=
x_0\left(\frac{1}{\beta-\beta_\infty}-\frac{1}{\beta}\right)
+ \<h_\infty, x_1\>
$$
so that we have
$$
\Gamma^*_\infty(t)-\widetilde\Gamma_\infty(t)
=
\left[
x_0\left(\frac{1}{\beta-\beta_\infty}-\frac{1}{\beta}\right)
+ \<h_\infty, x_1\>
\right]
e^{\left(r+\delta +\frac{\kappa^\top \kappa}{\gamma}(1-(2\gamma)^{-1})
- f_{\infty}^{-1}\big(1+\delta k^{-b}\big)\right)t
+\frac{\kappa^\top}{\gamma}Z(t)}.
$$
This means that, when $\beta_\infty\ge0$
and $\<h_\infty, x_1\>\ge 0$
(which are both true when $\phi (s) \ge 0$
for every $s \in [-d,0]$),
the total capital and hence the optimal consumption level $c^*$ and bequest target $B^*$
are larger than what they would be in the non path-dependent case (i.e. when $\phi\equiv 0$). This is a consequence of the predictable, past component of labor income shaping human capital and hence total wealth.
The situation is less clear cut for the risky asset allocation $\theta^*$, as
there is a complex interplay between risk preferences and financial market parameters. Indeed,
denoting by
$\Theta_\phi$ the feedback map associated with  $\theta^*$ (see \eqref{EQ_DEF_FEEDBACK_MAP}), we have
\begin{align}\label{OPTIMAL_STRATEGIES_RET_INF2newbis}
\Theta_\phi(w,x)-\Theta_0(w,x)
=(\sigma^\top)^{-1}\left[
\left(\frac{1}{\beta-\beta_\infty}-\frac1\beta\right)
x_0 \left(\frac\kappa\gamma - \sigma_y\right)
+ \<h_\infty, x_1\>\frac\kappa\gamma
\right].
\end{align}
The result suggests a very rich set of empirical predictions on risky asset allocations depending on risk preferences, financial market parameters, and the relative contribution of the past vs. future component of human wealth. In the special case of $\beta_{\infty}=0$, for example, we have that  the wedge between $\Theta_\phi$ and $\Theta_0$ is entirely driven by the past component of human capital, as the negative hedging demand appearing in both  $\Theta_\phi$ and $\Theta_0$ only depends on the present component of human capital, and not on the capitalized market value of the labor income's past trajectory. The latter can tilt the asset allocation above or below the baseline optimum resulting in the case of no path dependency.

It is clear that our results could provide even richer empirical predictions in more realistic settings. The introduction of a fixed retirement date,\footnote{See \cite{DYBVIG_LIU_JET_2010} for a model with endogenous retirement date but without path-dependency.} for example, would allow the relative importance of the
past vs. future component of labor income to change as the retirement date approaches. This could generate a hump shaped pattern in the risky asset allocation, which would be consistent with empirical evidence often treated as a puzzle or more recently explained by assuming stock prices to be cointegrated with labor income
(\cite{BENZONI_ET_AL_2007}). The model discussed here could then offer an interesting way to reconcile theory and empirical observation within a tractable setting. The solution of the finite horizon version of the model is the object of current further research.

%\color{black}

%\section{Conclusion}
%[TO BE COMPLETED]

\section*{Acknowledgments}
The authors are grateful to Sara Biagini, Salvatore Federico, Beniamin Goldys, Margherita Zanella for useful comments and suggestions.
{The authors are also grateful to two anonymous referees whose careful scrutiny helped to improve the paper.}
%to provide them the proof of Lemma \ref{LEMMA_DELAY_DET_EQUATION},
% and thus of Lemma \ref{LEMMA_LIMIT_INFTY_XI_G_H_X_INF_RET},
%and to Salvatore Federico for fruitful discussions.

%\thanks{ciao}

%\end{acknowledgment}

%CONTROLLA

\end{document}